\documentclass[a4paper]{amsart}
\usepackage[ansinew]{inputenc}
\usepackage[T1]{fontenc}
\usepackage{longtable}
\usepackage{xcolor}
\usepackage{mathrsfs}
\usepackage{amssymb,amsmath,amsthm}
\usepackage[matrix, arrow, curve]{xy}
\usepackage{tikz-cd}
\usepackage{stmaryrd}
\usepackage{mathtools}
\usepackage{extarrows}
\usepackage{hhline}
\usepackage{hyperref}
\usepackage{thmtools} %Cref fix 2025
\usepackage[nameinlink]{cleveref}
\usetikzlibrary{quotes,babel, angles}

\newcommand{\op}{\textnormal{op}}
\newcommand{\Hom}{\textnormal{Hom}}
\newcommand{\Ext}{\textnormal{Ext}}
\renewcommand{\mod}{\textnormal{mod}\,}

\newcommand{\Mod}{\textnormal{Mod}\,}
\newcommand{\im}{\textnormal{Im}\,}

\newcommand{\gen}{\textnormal{gen}\,}
\newcommand{\cogen}{\textnormal{cogen}\,}
\newcommand{\Ab}{\textnormal{Ab}}
\newcommand{\fp}{\textnormal{fp}\,}

\newcommand{\Ind}{\textnormal{Zsp}\,}
\newcommand{\Inj}{\textnormal{Inj}\,}

\numberwithin{equation}{section} \theoremstyle{plain}
\newtheorem*{thm*}{Theorem}
\newtheorem{thm}{Theorem}
\numberwithin{thm}{section}

\newtheorem{coro}[thm]{Corollary}
\newtheorem*{coro*}{Corollary}
\newtheorem{conj}[thm]{Conjecture}
\newtheorem*{conj*}{Conjecture}
\newtheorem{lem}[thm]{Lemma}
\newtheorem*{lem*}{Lemma}
\newtheorem{prop}[thm]{Proposition}
\newtheorem*{prop*}{Proposition}
\newtheorem{rem}[thm]{Remark}
\newtheorem*{rem*}{Remark}
\newtheorem{exa}[thm]{Example}
\newtheorem*{exa*}{Example}
\newtheorem{df}[thm]{Definition}
\newtheorem*{df*}{Definition}

\newtheorem*{ques*}{Question}

\newtheorem*{construction*}{Construction}
\newtheorem*{ack*}{ACKNOWLEDGEMENTS}

\begin{document}

\title{Infinite $\tau$-tilting theory}
\author{Kevin Schlegel}
\date{}

\address{Faculty of Mathematics, Bielefeld University, 33615 Bielefeld, Germany}
\email{schlegel@math.uni-bielefeld.de}

%\date{\today}

\begin{abstract} 
We classify torsion pairs in an essentially small abelian category through cosilting subsets of the Ziegler spectrum of the ind-completion of the abelian category. For Artin algebras, this classification is reformulated as an infinite analog of $\tau$-tilting theory, where torsion classes correspond to support $\tau$-tilting subsets of the Ziegler spectrum and torsion-free classes correspond to support $\tau^-$-tilting subsets. We further express the classification through ideals of the module category, thereby obtaining a formulation that involves finite length modules only. The developed theory is applied to study generic bricks and generic $\tau^-$-rigid modules, in particular for tame algebras, for which we show that these classes of modules coincide. We also recover a result of Bautista, P\' erez and Salmer\'on stating that a tame algebra admits infinitely many bricks of a fixed dimension if and only if there exists a generic brick. Finally, we prove that every algebra whose Krull-Gabriel dimension is defined satisfies the brick version of the second Brauer-Thrall conjecture.
\end{abstract}
\maketitle \section*{Introduction}
Torsion pairs are important objects in the representation theory of finite dimensional algebras (more generally, Artin algebras $A$). In the seminal work of Adachi, Iyama and Reiten \cite{air}, $\tau$-tilting theory was developed to study the functorially\linebreak finite torsion pairs inside the category $\mod A$ of finite length $A$-modules. We build on recent developments by Angeleri H\"ugel, Laking and Sentieri \cite{HLS, HLS2}, who make use of the Ziegler spectrum to study arbitrary torsion pairs in $\mod A$. Originating in model theory \cite{Ziegler}, the Ziegler spectrum $\Ind A$ is a topological space that has become an effective tool within representation theory. The points of $\Ind A$ are the indecomposable pure-injective $A$-modules, which include all indecomposable finite length $A$-modules as discrete points. Our main contribution is to show that the Auslander-Reiten translation $\tau$, which lies at the heart of $\tau$-tilting theory, can also be used to study arbitrary torsion pairs through the Ziegler spectrum, leading to what we call infinite $\tau$-tilting theory.

The Auslander-Reiten translation $\tau$ is usually considered only for finite length $A$-modules, but it was shown by Krause \cite{Krause0} that  $\tau$ naturally extends to infinite length $A$-modules and induces a homeomorphism $\tau \colon \Ind A \setminus \mathrm{Proj}\,A \rightarrow \Ind A \setminus \Inj A$. Building on this extension, we introduce the following notions, which may be viewed as infinite analogs of the corresponding concepts in classical $\tau$-tilting theory.
\begin{itemize}
    \item[(1)] A subset $U \subseteq \Ind A$ is $\tau$-\emph{rigid} if $\Hom_A^\mathrm{fin}(U, \tau U) = 0$. That is, there are no non-zero morphisms $X \to \tau Y$ with finite length image for every $X,Y \in U$. 
    \item[(2)] We call $(U,V)$ a $\tau$\emph{-rigid} pair if $U \subseteq \Ind A$ is $\tau$-rigid and $V \subseteq \mathrm{Proj}\,A$ fulfills $\Hom_A(V, U )= 0$. There is a natural order on $\tau$-rigid pairs given by $(U,V)\leq (U',V')$ if $U \subseteq U'$ and $V\subseteq V'$.
    \item[(3)] A subset $U \subseteq \Ind A$ is \emph{support $\tau$-tilting} if there exists a maximal $\tau$-rigid pair $(U,V)$. 
\end{itemize}

The point of $\tau$-rigid subsets $U \subseteq \Ind A$ is that they give rise to torsion classes inside $\mod A$. More precisely, the collection $\gen U$ of all finite length quotients of direct sums of modules in $U$ is a torsion class. One may define dual notions, which are related to torsion-free classes, using the inverse Auslander-Reiten translation $\tau^-$ instead. This leads to the following classifications.\\
\\
\textbf{Theorem A.} (\Cref{taumain}, \Cref{main2}) \emph{There are one-to-one correspondences
\begin{align*}
    \begin{Bmatrix}
        \text{torsion classes}\\
        \mathcal{T}\subseteq \mathrm{mod}\,A
    \end{Bmatrix} &\longleftrightarrow 
    \begin{Bmatrix}
        \text{support $\tau$-tilting subsets}\\
        {U}\subseteq \Ind A
    \end{Bmatrix}\\
    \begin{Bmatrix}
        \text{torsion-free classes}\\
        \mathcal{F}\subseteq \mathrm{mod}\,A
    \end{Bmatrix} &\longleftrightarrow 
    \begin{Bmatrix}
        \text{support $\tau^-$-tilting subsets}\\
        {U}'\subseteq \Ind A
    \end{Bmatrix}
\end{align*}
given by $U \mapsto  \gen U$ and $U' \mapsto \cogen U'$. Moreover, support $\tau$-tilting subsets and support $\tau^-$-tilting subsets are closed sets of $\Ind A$.}\\

The theorem above involves modules that are possibly of infinite length, which makes the situation more complicated than in classical $\tau$-tilting theory. To address this difficulty, we consider ideals of $\mathrm{mod}\,A$, that is, collections of morphisms that satisfy closure properties similar to ideals of rings. For $U \subseteq \Ind A$ let $\langle U \rangle$ be the collection of all morphisms in $\mod A$ factoring through a product of modules in $U$. Krause showed that the assignment $U \mapsto \langle U \rangle$ induces a one-to-one correspondence between closed sets $U \subseteq \Ind A$ and certain ideals $\mathcal{I}$ of $\mod A$, see \cite{Krause}. Now $U$ is $\tau$-rigid if and only if the associated ideal $\mathcal{I}$ satisfies $\tau \mathcal{I} \circ \mathcal{I} = 0$, where $\tau \mathcal{I}$ may be defined via the Auslander-Reiten translation on morphisms. Such an ideal is called \emph{$\tau$-rigid}. This naturally gives rise to the notions of \emph{support $\tau$-tilting ideals} and \emph{support $\tau^{-}$-tilting ideals}, leading to the following correspondences.\\
\\
\textbf{Theorem B.} (\Cref{mainideal}, \Cref{mainideal2}) \emph{There are one-to-one correspondences
\begin{align*}
    \begin{Bmatrix}
        \text{torsion classes}\\
        \mathcal{T}\subseteq \mathrm{mod}\,A
    \end{Bmatrix} &\longleftrightarrow 
    \begin{Bmatrix}
        \text{support $\tau$-tilting ideals}\\
        \mathcal{I}\text{ of } \mod A
    \end{Bmatrix}\\
    \begin{Bmatrix}
        \text{torsion-free classes}\\
        \mathcal{F}\subseteq \mathrm{mod}\,A
    \end{Bmatrix} &\longleftrightarrow 
    \begin{Bmatrix}
        \text{support $\tau^-$-tilting ideals}\\
        \mathcal{I}'\text{ of } \mod A
    \end{Bmatrix}
\end{align*}
given by $\mathcal{I} \mapsto  \gen \mathcal{I}$ and $\mathcal{I}' \mapsto \cogen \mathcal{I}'$.}\\

Here $\gen \mathcal{I}$ denotes the collection of all $X\in \mod A$ that admit an epimorphism $Y\to X$ in $\mathcal{I}$ and $\cogen \mathcal{I}'$ is defined dually. The following example illustrates the usefulness of the ideal approach. If $U \subseteq \Ind A$ is an infinite collection of finite length $A$-modules with closure $\bar{U}$, then the ideal $\langle \bar{U}\setminus U\rangle$ equals the collection of all morphisms that factor through a direct sum of modules in $U\setminus V$ for all finite subsets $V\subseteq U$. If $U$ is, in some sense, generically $\tau$-rigid, the ideal can easily be shown to be $\tau$-rigid and so $\bar{U}\setminus U$ is a non-empty $\tau$-rigid set consisting of infinite length modules. 

Another advantage of the ideal approach is that we can apply duality, as we are only dealing with finite length $A$-modules. This is not directly possible for infinite length $A$-modules. In fact, we first prove the $\tau^-$-part of Theorem A to deduce the $\tau^-$-part of Theorem B, implying the $\tau$-part of Theorem B via duality, which is used to show the $\tau$-part of Theorem A. Thus, the mentioned results are based on the $\tau^-$-part of Theorem A, which is a special instance of a more general classification. Namely, for an essentially small abelian category $\mathcal{C}$ we characterize torsion pairs through cosilting subsets of the Ziegler spectrum $\Ind \bar{\mathcal{C}}$ of the ind-completion $\bar{\mathcal{C}}$. For the definition of cosilting subsets, see Section 2. \\
\\
\textbf{Theorem C.} (\Cref{mai}) \emph{Let $\mathcal{C}$ be an essentially small abelian category. There is a one-to-one correspondence 
\begin{equation*}
    \begin{Bmatrix}
        \text{torsion-free classes}\\
        \mathcal{F}\subseteq \mathcal{C}
    \end{Bmatrix} \longleftrightarrow 
    \begin{Bmatrix}
        \text{cosilting subsets}\\
        {U}\subseteq \Ind \bar{\mathcal{C}}
    \end{Bmatrix}
\end{equation*}
given by $\mathcal{F} \mapsto \Inj \bar{\mathcal{F}}$ and $U \mapsto \cogen U$. Moreover, cosilting subsets are closed.}\\

For the module category of an Artinian ring, a similar classification as above was established by Angeleri H\"ugel, Laking and Sentieri \cite[Theorem A]{HLS}. Their approach is based on 2-term complexes of injective objects and the Ziegler spectrum of a derived category, while ours relies on a connection between exact structures and purity established in \cite{self2}. We naturally consider a torsion-free class $\mathcal{F}\subseteq \mathcal{C}$ as well as its direct limit closure $\bar{\mathcal{F}} =\varinjlim \mathcal{F}$ as an exact category, since they are extension-closed subcategories. The collection of indecomposable injective objects in the exact category $\bar{\mathcal{F}}$ then forms the associated cosilting subset $\Inj\bar{\mathcal{F}} \subseteq \bar{\mathcal{C}}$. 

Let us return to Artin algebras $A$, where  the cosilting subsets of $\Ind A$ coincide with the support $\tau^-$-tilting subsets. This gives support $\tau^{-}$-tilting subsets a natural interpretation in terms of injective objects in an exact category, an interpretation that is not available for support $\tau$-tilting subsets. Thus, the $\tau^{-}$-perspective is more natural when studying modules of infinite length. We focus on \emph{generic} modules, that is, indecomposable $A$-modules $X$ of infinite length that have finite length over $\mathrm{End}_A(X)$. Generic modules form closed points $\{X\} \subseteq \Ind A$ and we investigate them in terms of $\tau^-$-rigidity. This property is related to \emph{bricks}, which are $A$-modules  whose endomorphism ring is a division algebra. We apply the developed theory of infinite $\tau$-tilting theory to show the following result.\\
\\
\textbf{Theorem D.} (\Cref{tame}, \Cref{KG}) \emph{Let $A$ be a finite dimensional algebra over an algebraically closed field. Consider the following statements.}
\begin{itemize}
    \item[(1)] \emph{There are infinitely many non-isomorphic finite dimensional bricks.}
    \item[(2)] \emph{There are infinitely many non-isomorphic finite dimensional bricks of the same dimension.}
    \item[(3)] \emph{There exists a generic brick.}
\end{itemize}
\emph{If $A$ is tame, then {(2)} and {(3)} are equivalent, and if the Krull-Gabriel dimension of $A$ is defined, then all of the statements are equivalent.}\\

The equivalences of the statements in the above theorem were conjectured by Mousavand and Paquette \cite[Conjecture 1.1]{MP}, and the equivalence of (2) and (3) has already been proven for tame algebras by Bautista, P\' erez and Salmer\'on \cite{BPS} using matrix reduction techniques. The Krull-Gabriel dimension, introduced by \mbox{Geigle \cite{Geigle}}, is a measure for the complexity of the module category. A conjecture of Prest 
states that the Krull-Gabriel dimension of $A$ is defined if and only if $A$ is domestic 
\cite[Conjecture 9.1.15]{Prest}, which is known for several classes of algebras,  see \cite[Section 1]{Pastu} for a detailed discussion. Lastly, we show the following surprising connection between $\tau^-$-rigidity and bricks over tame algebras.\\
\\
\textbf{Theorem E.} (\Cref{weird}) For a tame finite dimensional algebra $A$ over an algebraically closed field, a generic module $X$ is a brick if and only if $\{X\}$  is $\tau^-$-rigid.\\
\\
\textbf{Acknowledgments.} This work was funded by the Deutsche Forschungsgemeinschaft (DFG, German Research Foundation) under the Walter Benjamin Programme -- Project number 580964112.

\section{Purity and exact structures}  
We will later use the fact that torsion-free classes naturally form exact categories and consider big objects in order to classify them. A main tool to achieve this is the theory of purity for finitely accessible categories with products and its connection with exact structures established in \cite{self2}. 

An additive category $\mathcal{A}$ is called \emph{finitely accessible} if $\mathcal{A}$ has filtered colimits, the subcategory $\fp \mathcal{A}$ of finitely presented objects is essentially small, and every object in $\mathcal{A}$ is a filtered colimit of finitely presented objects. Recall that $X\in \mathcal{A}$ is \emph{finitely presented} if $\Hom_{\mathcal{A}}(X,-)$ commutes with filtered colimits. For an essentially small additive category $\mathcal{C}$ its ind-completion $\bar{\mathcal{C}}$ is a finitely accessible category. This assignment induces the following bijection due to Crawley-Boevey.

\begin{thm}\label{Ind} \cite[Theorem 1.4]{Crawley-Boevey2} There exists a one-to-one correspondence, up to equivalence, between finitely accessible categories $\mathcal{A}$ and essentially small additive categories $\mathcal{C}$ with split idempotents, given by $\mathcal{A} \mapsto \fp \mathcal{A}$ and $\mathcal{C} \mapsto \bar{\mathcal{C}}$.
\end{thm}

Let $\mathcal{A}$ be a finitely accessible category. A filtered colimit of split exact sequences in $\mathcal{A}$ is a \emph{pure-exact sequence}. The pure-exact sequences form an exact structure and the injective objects are called \emph{pure-injective}. We are particularly interested in the case when $\mathcal{A}$ has products. Then, the isomorphism classes of pure-injective indecomposable objects form a topological space $\Ind \mathcal{A}$, the \emph{Ziegler spectrum} of $\mathcal{A}$. For more details, see for example \cite{Krause3}.

For an additive category $\mathcal{C}$ and a full additive subcategory $\mathcal{D} \subseteq \mathcal{C}$ a morphism $f\colon X \to C_X$ in $\mathcal{C}$ is a \emph{left $\mathcal{D}$-approximation} of $X$ if $C_X \in \mathcal{D}$ and $\Hom_{\mathcal{D}}(f,Y)$ is surjective for all $Y\in \mathcal{D}$. Moreover $\mathcal{D}$ is \emph{covariantly finite} if every $X\in \mathcal{C}$ admits a left $\mathcal{D}$-approximation.

\begin{prop}\label{covar}\cite[Proposition 4.13, Corollary 6.3]{Krause3} Let $\mathcal{A}$ be a finitely accessible category with products, $\mathcal{D} \subseteq  \fp \mathcal{A}$ a covariantly finite subcategory, and $\mathcal{B} = \varinjlim \mathcal{D}$ the closure of $\mathcal{D}$ in $\mathcal{A}$ under filtered colimits. Then $\mathcal{B}$ is a finitely accessible category with products such that $\fp \mathcal{B} = \mathcal{D}$ and $\Ind \mathcal{B} \subseteq \Ind \mathcal{A}$ is a closed subset.
\end{prop}

Next, we discuss the connection with exact structures. For the definition of exact structures and basic properties we refer to the exposition of B\"uhler \cite{Buehler}. We use the terminology of \emph{admissible short exact sequences} and  \emph{admissible monomorphisms}.

Following Positselski \cite{Positselski}, an exact structure on a finitely accessible category $\mathcal{A}$ is \emph{locally coherent} if every admissible short exact sequence is a filtered colimit of admissible short exact sequences $    0 \to X \to Y \to Z \to 0$ with $X,Y,Z \in \fp \mathcal{A}$. For example, the pure-exact sequences form a locally coherent exact structure. The following connects locally coherent exact structures with the Ziegler spectrum.

\begin{thm}\label{connect}\cite[Proposition 2.12, Theorem 3.1]{self2} Let $\mathcal{A}$ be a finitely accessible category with products and consider a locally coherent exact structure on $\mathcal{A}$. Every $X\in \mathcal{A}$ admits an admissible monomorphism $X\to Q$ where $Q$ is a product of indecomposable injectives. Moreover, the isomorphism classes of indecomposable injectives form a closed set in $\Ind \mathcal{A}$.
\end{thm}

\section{Torsion pairs via cosilting subsets}

Let $\mathcal{C}$ be an essentially small abelian category. A \emph{torsion pair} $(\mathcal{T}, \mathcal{F})$ of $\mathcal{C}$ is a pair of full subcategories such that $\Hom_{\mathcal{C}}(\mathcal{T}, \mathcal{F}) = 0$ and every $X\in \mathcal{C}$ admits a short exact sequence $0 \to X' \to X \to X'' \to 0$ with $X' \in \mathcal{T}$ and $X'' \in \mathcal{F}$. Moreover $\mathcal{T}$ is a \emph{torsion class} and $\mathcal{F}$ a \emph{torsion-free class}. Recall that a torsion pair is uniquely determined by the torsion-free class. For our classification, we restrict to torsion-free classes and use the following characterization by Gentle and Todorov.

\begin{prop}\label{Char}\cite[Proposition 3.9']{GT} Let $\mathcal{C}$ be an abelian category and $\mathcal{F}\subseteq \mathcal{C}$ a full subcategory. Then $\mathcal{F}$ is a torsion-free class if and only if $\mathcal{F}$ is covariantly finite and closed under extensions as well as subobjects.
\end{prop}
%\begin{proof} It is well-known that a torsion-free class is closed under extensions as well as subobjects and the canonical short exact sequences of the corresponding torsion pair provide left $\mathcal{F}$-approximations. 

%Now suppose that $\mathcal{F}$ is covariantly finite and closed under extensions as well as subobjects. Set $\mathcal{T} = \{X \in \mathcal{C}\mid \Hom_{\mathcal{C}}(X, \mathcal{F}) = 0\}$. Then $\Hom_{\mathcal{C}}(\mathcal{T}, \mathcal{F}) = 0$ and it is left to show the existence of the canonical short exact sequences. For $X\in \mathcal{C}$ let $X' \leq X$ be the intersection of all kernels of morphisms $X \to Y$ with $Y \in \mathcal{F}$. This is well-defined, since $X'$ equals the kernel of a left $\mathcal{F}$-approximation $f\colon X\to C_X$. Consider the short exact sequence $0 \to X' \to X \to X'' \to 0$ with $X'' = \im f \in \mathcal{F}$. We must verify that $X' \in \mathcal{T}$. Let $X' \to Y$ be a morphism with $Y \in \mathcal{F}$. Forming a pushout $P$ yields a commutative diagram
%\begin{equation*}
%    \begin{tikzcd}
%        0 \arrow[r] & X' \arrow[r] \arrow[d] & %X \arrow[r] \arrow[d] & X'' \arrow[r] \arrow[d, equals] & 0 \\ 
%        0 \arrow[r] & Y \arrow[r]  & P \arrow[r] & X'' \arrow[r] & 0 
%    \end{tikzcd}
%\end{equation*}
%with exact rows. Since $\mathcal{F}$ is closed under extensions, it follows that $P \in \mathcal{F}$ and by definition of $X'$ the composition $X' \to X \to P$ is zero. Hence, also $X' \to Y$ is a zero morphism and $X' \in \mathcal{T}$.
%\end{proof}

From now on we fix an essentially small abelian category $\mathcal{C}$. By \Cref{Ind} the ind-completion $\bar{\mathcal{C}}$ is finitely accessible and we may identify $\fp \bar{\mathcal{C}} = \mathcal{C}$. In fact $\bar{\mathcal{C}}$ is a Grothendieck category \cite[Satz 1.5]{Breitsprecher} so it has products. In particular, the Ziegler spectrum $\Ind \bar{\mathcal{C}}$ is defined.

For a torsion-free class ${\mathcal{F}} \subseteq {\mathcal{C}}$ let $\bar{\mathcal{F}} = \varinjlim \mathcal{F}$ be the closure of $\mathcal{F}$ in $\bar{\mathcal{C}}$ under filtered colimits. Since $\mathcal{F}$ is covariantly finite in $\mathcal{C}$ it follows from \Cref{covar} that $\bar{\mathcal{F}}$ is a finitely accessible category with products and $\fp \bar{\mathcal{F}} = \mathcal{F}$. Moreover, the Ziegler spectrum $\Ind \bar{\mathcal{F}}$ is a closed subset of $\Ind \bar{\mathcal{C}}$. Note that again $\bar{\mathcal{F}}\subseteq \bar{\mathcal{C}}$ is a torsion-free class \cite[Lemma 4.4]{Crawley-Boevey2} and in particular $\bar{\mathcal{F}}$ is closed under extensions. Thus, there is a canonical exact structure on $\bar{\mathcal{F}}$ given by all short exact sequences $0\to X\to Y \to Z\to 0$ in $\bar{\mathcal{C}}$ with $X,Y,Z\in\bar{ \mathcal{F}}$, see for example \cite[Lemma 10.20]{Buehler}. In this sense, we regard $\bar{\mathcal{F}}$ as an exact category. 

\begin{lem} Let $\mathcal{F}$ be a torsion-free class in $ \mathcal{C}$. Then $\bar{\mathcal{F}}$ is a locally coherent exact category.
\end{lem}
\begin{proof} Let $0 \to X \to Y \to Z \to 0$ be a short exact sequence in $\bar{\mathcal{C}}$ with $X,Y,Z \in \bar{\mathcal{F}}$. Since $\fp \bar{\mathcal{F}} = \mathcal{F}$ there exists a  filtered colimit $Z = \varinjlim Z_i$ with $Z_i \in \mathcal{F}$. Consider the commutative diagram of short exact sequences
\begin{equation*}
\begin{tikzcd}
    0 \arrow[r] & X \arrow[r] & Y \arrow[r] & Z \arrow[r] & 0 \\
    0 \arrow[r] & X \arrow[u, equals] \arrow[r] & P_i \arrow[u] \arrow[r] & Z_i \arrow[r]\arrow[u] & 0
\end{tikzcd}
\end{equation*}
where the right square is a pullback diagram. Then $P_i \in \bar{\mathcal{F}}$ and there are filtered colimits $P_i = \varinjlim P_{ij}$ with $P_{ij} \in \mathcal{F}$. Let $K_{ij} \in \mathcal{F}$ be the kernel and $I_{ij} \in \mathcal{F}$ the image of the composition $P_{ij} \to P_i \to  Z_i$. Again, there are commutative diagrams of short exact sequences
\begin{equation*}
\begin{tikzcd}
    0 \arrow[r] & X \arrow[r] & P_{i} \arrow[r] & Z_i \arrow[r] & 0 \\
    0 \arrow[r] & K_{ij} \arrow[u] \arrow[r] & P_{ij} \arrow[u] \arrow[r] & I_{ij} \arrow[r]\arrow[u] & 0.
\end{tikzcd}
\end{equation*}
Since filtered colimits are exact inside the Grothendieck category $\bar{\mathcal{C}}$ it follows that the filtered colimit over $j$ of the short exact sequences in the second row recovers the short exact sequence in the first row. Taking the filtered colimit over $i$ we then recover the original short exact sequence. Note that $0 \to K_{ij} \to P_{ij} \to I_{ij} \to 0$ is an admissible short exact sequence inside the exact category $\bar{\mathcal{F}}$ that only involves finitely presented objects. Thus $\bar{\mathcal{F}}$ is locally coherent.
\end{proof}

For a torsion-free class $\mathcal{F}\subseteq \mathcal{C}$ let $\Inj \bar{\mathcal{F}}$ denote the set of isomorphism classes of indecomposable injective objects inside the locally coherent exact category $\bar{\mathcal{F}}$. By \Cref{connect} it follows that $\Inj \bar{\mathcal{F}}$ is a closed subset of $\Ind \bar{\mathcal{F}}$ and we already observed that $\Ind \bar{\mathcal{F}}$ is a closed subset of $\Ind \bar{\mathcal{C}}$. In combination, we obtain the following result, which may be seen as a generalization of \cite[Corollary 3.3]{HLS}.

\begin{coro}\label{clsub} Let $\mathcal{F}$ be a torsion-free class in $\mathcal{C}$. Then $\Inj \bar{\mathcal{F}}$ is a closed subset of $\Ind \bar{\mathcal{C}}$.
\end{coro}

The goal will be to classify torsion-free classes $\mathcal{F} \subseteq \mathcal{C}$ via the associated closed subset $\Inj \bar{\mathcal{F}} \subseteq \Ind \bar{\mathcal{C}}$. This requires the following definition. (Note that for the mo- dule category of an Artinian ring, cogen-rigid pairs coincide with rigid pairs in \cite{HLS}.)  

\begin{df}\label{defimp}\rm For a set $U$ of objects in $\bar{\mathcal{C}}$ we denote by $\cogen U \subseteq \mathcal{C}$ the full subcategory of all $X\in \mathcal{C}$ that admit a monomorphism $X \to Y$ in $\bar{\mathcal{C}}$ where $Y$ is a product of objects in $U$.
\begin{itemize}
    \item[(1)] 
A subset $U \subseteq \Ind \bar{\mathcal{C}}$ is \emph{cogen-rigid} if
$\mathrm{Ext}_{\bar{\mathcal{C}}}^1 (\cogen U, U) = 0$, every $X\in U$ is a filtered colimit of objects in $\cogen U$, and $\cogen U$ is covariantly finite.
    \item[(2)] We call $(U,V)$ a \emph{cogen-rigid} pair if $U \subseteq \Ind \bar{\mathcal{C}}$ is cogen-rigid and $V\subseteq \Inj \bar{\mathcal{C}}$ fulfills $\Hom_{\bar{\mathcal{C}}}(X, V) = 0$ for all products $X$ of objects in $U$. There is a partial\linebreak order on cogen-rigid pairs given by $(U,V)\leq (U',V')$ if $U \subseteq U'$ and $V \subseteq V'$.
    \item[(3)] A subset $U\subseteq \Ind \bar{\mathcal{C}}$ is \emph{cosilting} if there is a maximal cogen-rigid pair $(U,V)$. 
\end{itemize}
\end{df}

\begin{rem}\label{length}\rm Some of the properties for a subset $U \subseteq \Ind \bar{\mathcal{C}}$ to be cogen-rigid hold automatically under suitable finiteness conditions on $\mathcal{C}$.
\begin{itemize}
    \item[(1)] Suppose that $\mathcal{C}$ is Noetherian. That is, the lattice of subobjects of any object in $\mathcal{C}$ fulfills the ascending chain condition. Then every $X \in U$ is a filtered colimit of objects in $\cogen U$. Namely $X = \varinjlim Y$ where the colimit goes over all finitely generated subobjects $Y \subseteq X$. Being Noetherian implies that $Y$ is finitely presented \cite[Proposition 11.2.5]{Krause4} and so each $Y$ is contained in $\cogen U$.
      \item[(2)] Suppose that $\mathcal{C}$ is Artinian. That is, the lattice of subobjects of any object in $\mathcal{C}$ fulfills the descending chain condition. Then $\cogen U$ is covariantly finite. For $X\in \mathcal{C}$ a left $\cogen U$-approximation $X\to C_X$ is given by the image of the canonical map $f\colon X \to \prod Y$ where the product goes over all $g\colon X \to Y$ with $Y \in \cogen U$. Being Artinian implies that $\ker f = \bigcap \ker g$ for a finite intersection and so $\im f \in \cogen U$.
      \item[(3)] If $\mathcal{C}$ is a length category, equivalently $\mathcal{C}$ is Noetherian and Artinian, then by the previous observations a subset $U \subseteq \Ind \bar{\mathcal{C}}$ is cogen-rigid if and only if $\mathrm{Ext}_{\bar{\mathcal{C}}}^1 (\cogen U, U) = 0$. Moreover, in this case $(U,V)$ is a cogen-rigid pair if and only if $\Hom_{\bar{\mathcal{C}}}(U,V) = 0$ with $V\subseteq \Inj \bar{\mathcal{C}}$. This follows from the fact that any product $\prod X$ in $\bar{\mathcal{C}}$ equals the union of its finite length subobjects, which are isomorphic to subobjects of finite direct sums $\bigoplus X\leq \prod X$.
\end{itemize}
\end{rem}

The definition of cogen-rigid subsets $U \subseteq \Ind \bar{\mathcal{C}}$ is essentially made such that they give rise to torsion-free classes, see the following proposition.

\begin{prop}\label{keylem} Let $U \subseteq \Ind \bar{\mathcal{C}}$ be cogen-rigid and $\mathcal{F} \subseteq \mathcal{C}$ a torsion-free class. 
\begin{itemize}
    \item[(1)] The subcategory $\cogen U \subseteq \mathcal{C}$ is a torsion-free class and $U \subseteq \Inj \overline{\cogen U}$.
    \item[(2)] The subset $\Inj \bar{\mathcal{F}} \subseteq \Ind \bar{\mathcal{C}}$ is cogen-rigid and $\cogen \Inj \bar{\mathcal{F}} = \mathcal{F}$.
\end{itemize}
\begin{proof} (1) Clearly $\cogen U$ is closed under subobjects. By definition of cogen-rigid subsets $\cogen U$ is covariantly finite. We are left to show that $\cogen U$ is closed under extensions to be a torsion-free class, see \Cref{Char}. Consider a short exact sequence $0 \to X \to Y \to Z \to 0$ in $\mathcal{C}$ with $X,Z \in \cogen U$. Let $X \to \prod A$ be a monomorphism in $\bar{\mathcal{C}}$ where the product only involves objects in $U$. Taking a pushout $P$ yields a commutative diagram
\begin{equation*}
    \begin{tikzcd}
        0 \arrow[r] & \prod A \arrow[r] & P \arrow[r] & Z \arrow[r] & 0\\
        0\arrow[r] & X \arrow[r] \arrow[u] & Y\arrow[u] \arrow[r] & Z \arrow[r] \arrow[u, equals] &  0 
    \end{tikzcd}
\end{equation*}
with exact rows. Note that the canonical map $\Ext_{\bar{\mathcal{C}}}^1 (Z, \prod A) \to \prod \Ext_{\bar{\mathcal{C}}}^1 (Z, A)$ is always a monomorphism. Because $U$ is cogen-rigid, $\mathrm{Ext}_{\bar{\mathcal{C}}}^1 (Z,A) = 0$ and so the short exact sequence in the first row splits. It follows that $P \cong Z \oplus \prod A$ and since $Z \in \cogen U$ also $Y \in \cogen U$.

Next, we prove that $U \subseteq \Inj \bar{\mathcal{G}}$ for $\mathcal{G} = \cogen U$. By definition of a cogen-rigid subset $U$ is contained in  $\bar{\mathcal{G}}$. It is left to show that every $X \in U$ is injective in the exact category $\bar{\mathcal{G}}$. Let $0 \to X \to Y \to Z \to 0$ be a short exact sequence in $\bar{\mathcal{C}}$ with $Z \in \bar{\mathcal{G}}$ and write $Z = \varinjlim Z_i$ for $Z_i \in \cogen U$. Taking pullbacks $P_i$ yields commutative diagrams
\begin{equation*}
    \begin{tikzcd}
        0 \arrow[r] & X \arrow[r] & Y \arrow[r] & Z \arrow[r] & 0\\
        0\arrow[r] & X \arrow[r] \arrow[u, equals] & P_i \arrow[u] \arrow[r] & Z_i \arrow[r] \arrow[u] &  0 
    \end{tikzcd}
\end{equation*}
such that the short exact sequence in the first row is a filtered colimit of the short exact sequences in the second row, which split since $\Ext_{\bar{\mathcal{C}}}^1 (Z_i, X ) = 0$. It follows that $0 \to X \to Y \to Z \to 0$ is pure-exact so it also splits as $X$ is pure-injective. Thus $X$ is injective in $\bar{\mathcal{G}}$.

(2) First $\Inj \bar{\mathcal{F}} \subseteq \Ind \bar{\mathcal{C}}$ by \Cref{clsub} and $\mathcal{F} \subseteq \cogen \Inj \bar{\mathcal{F}}$ by \Cref{connect}. The second inclusion is an equality because $\mathcal{F} = \bar{\mathcal{F}} \cap \mathcal{C} \supseteq \cogen \Inj \bar{\mathcal{F}}$. Using this equality, it follows that $\Inj \bar{\mathcal{F}}$ is cogen-rigid since clearly  $\Ext^1_{\bar{\mathcal{C}}}(\mathcal{F},\Inj \bar{\mathcal{F}}) = 0$,  $\Inj \bar{\mathcal{F}} \subseteq \bar{\mathcal{F}}$ and $\mathcal{F}$ is covariantly finite by \Cref{Char}.
\end{proof}
\end{prop}

In order to assign a cogen-rigid subset $U \subseteq \Ind \bar{\mathcal{C}}$ to a torsion-free class $\mathcal{F}\subseteq \mathcal{C}$ in a unique way, we must restrict to cosilting subsets. With the following theorem we achieve our goal and obtain a classification of torsion pairs for essentially small abelian categories. Note that a similar correspondence has been established by Angeleri H\" ugel, Laking and Sentieri \cite[Theorem A]{HLS} for the case $\mathcal{C} = \mod R$ where $R$ is an Artinian ring. Whereas our approach relies on the connection between purity and exact structures, theirs is based on 2-term complexes of injective objects and the Ziegler spectrum of the derived category.

\begin{thm}\label{mai} Let $\mathcal{C}$ be an essentially small abelian category. There is a one-to-one correspondence 
\begin{equation*}
    \begin{Bmatrix}
        \text{torsion-free classes}\\
        \mathcal{F}\subseteq \mathcal{C}
    \end{Bmatrix} \longleftrightarrow 
    \begin{Bmatrix}
        \text{cosilting subsets}\\
        {U}\subseteq \Ind \bar{\mathcal{C}}
    \end{Bmatrix}
\end{equation*}
given by $\mathcal{F} \mapsto \Inj \bar{\mathcal{F}}$ and $U \mapsto \cogen U$. Moreover, cosilting subsets are closed.
\end{thm}
\begin{proof} By \Cref{keylem} the assignments yield $\mathcal{F} \mapsto \Inj \bar{\mathcal{F}} \mapsto \cogen \Inj \bar{\mathcal{F}} = \mathcal{F}$ for every torsion-free class $\mathcal{F} \subseteq \mathcal{C}$ and $U \mapsto \cogen U \mapsto \Inj \overline{\cogen U} \supseteq U$ for every cosilting subset $U \subseteq \Ind \bar{\mathcal{C}}$. Further $\Inj \bar{\mathcal{F}}\subseteq \Ind \bar{\mathcal{C}}$ is closed by  \Cref{clsub}. It is left to show that $\Inj \bar{\mathcal{F}}$ is cosilting and not only cogen-rigid, and $U = \Inj \overline{\cogen U}$. We begin with the latter.  

Let $V$ be the collection of all $Y\in \Inj \bar{\mathcal{C}}$ with $\Hom_{\bar{\mathcal{C}}}(X,Y) = 0$ for every product $X$ of objects in $U$. By definition of being cosilting $(U,V)$ is a maximal cogen-rigid pair. One easily checks that $\Hom_{\bar{\mathcal{C}}}(\cogen U, V) = 0$ which implies $\Hom_{\bar{\mathcal{C}}}(\overline{\cogen U}, V) = 0$. It follows that $(U,V) \leq (\Inj \overline{\cogen U}, V)$ as cogen-rigid pairs and so $U = \Inj \overline{\cogen U}$ by the maximality condition.

We know that $\Inj \bar{\mathcal{F}}$ is cogen-rigid and must show that $(\Inj \bar{\mathcal{F}}, V)$ is a maximal cogen-rigid pair, where $V$ is the collection of all $Y\in \Inj \bar{\mathcal{C}}$ with $\Hom_{\bar{\mathcal{C}}}(X,Y) = 0$ for all products $X$ of objects in $\Inj \bar{\mathcal{F}}$ or equivalently $\Hom_{\bar{\mathcal{C}}}(\mathcal{F}, Y) = 0$. Let $(U',V')$ be a cogen-rigid pair with $\Inj \bar{\mathcal{F}} \subseteq U'$ and $V \subseteq V'$. Then $\mathcal{F}' = \cogen U'$ is a torsion-free class with $\Inj \bar{\mathcal{F}} \subseteq \Inj \bar{\mathcal{F}'}$ by \Cref{keylem}. By \cite[Theorem 4.2]{Krause2} the condition $V \subseteq V'$ implies that the smallest Serre subcategory $\mathcal{S} \subseteq \mathcal{C}$ containing $\mathcal{F}$ also contains $\mathcal{F}'$. Note that $\mathcal{S}$ consists of all objects in $\mathcal{C}$ that are filtered by quotients of objects in $\mathcal{F}$. It follows that $\mathcal{F}' = \bigcup_{n\geq 0} \mathcal{F}'_n$ where $\mathcal{F}'_n \subseteq \mathcal{F}'$ is the full subcategory of all objects filtered by $n$-many quotients of objects in $\mathcal{F}$. We will show that $\mathcal{F}'_n \subseteq \mathcal{F}$ for all $n$ but before we make the following observation. 

For $X\in \bar{\mathcal{C}}$ consider the short exact sequence $0 \to tX \to X \to X/tX \to 0$ where $tX$ is contained in the torsion class corresponding to $\bar{\mathcal{F}}$ and $X/tX \in \bar{\mathcal{F}}$. Then $X\to X/tX$ is a left $\bar{\mathcal{F}}$-approximation and $L\colon \bar{\mathcal{F}'}\to \bar{\mathcal{F}}, X \mapsto X/tX$ is a left adjoint of the inclusion $\iota \colon \bar{\mathcal{F}} \to \bar{\mathcal{F}'}$. The condition $\Inj \bar{\mathcal{F}} \subseteq \Inj \bar{\mathcal{F}'}$ implies that $L$ is exact by \Cref{exactlem}. It follows that for every epimorphism $f\colon Y\to X$ with $X \in \mathcal{F}'$ and $Y\in \mathcal{F}$ already  $X\in \mathcal{F}$. Indeed, applying the exact functor $L$ to the short exact sequence $0 \to \ker f \to Y\to X\to 0$ shows that $LX = X$.

Now, for $X \in \mathcal{F}_n'$ there is a short exact sequence $0 \to X' \to X \to X'' \to 0$ with $X' \in \mathcal{F}'_{n-1}$ and $X''$ is the quotient of some $Y \in \mathcal{F}$. By induction we may assume that $X' \in \mathcal{F}$. Taking a pullback $P$ yields a commutative diagram
\begin{equation*}
\begin{tikzcd}
    0 \arrow[r] & X' \arrow[d, equals] \arrow[r] & P \arrow[r] \arrow[d] & Y \arrow[r] \arrow[d] & 0 \\
    0 \arrow[r] & X' \arrow[r] & X \arrow[r] & X'' \arrow[r] & 0
\end{tikzcd}
\end{equation*}
with exact rows. Since $P \to X$ is epic and $P \in \mathcal{F}$ it follows from the previous observation that $X\in \mathcal{F}$. In total, we have shown that $\mathcal{F}' = \bigcup_{n\geq 0 }\mathcal{F}'_n = \mathcal{F}$ and in particular $\Inj \bar{\mathcal{F}} = \Inj \bar{\mathcal{F}'}$. Hence $\Inj \bar{\mathcal{F}}$ is cosilting.
\end{proof}

The proof of \Cref{mai} is based on viewing torsion-free classes naturally as exact categories. It uses the following lemma, which is well-known in the abelian case and generalizes to the exact context.

\begin{lem}\label{exactlem} Let $\mathcal{A}, \mathcal{B}$ be exact categories and  $L : \mathcal{A} \rightleftarrows \mathcal{B} : R$ an adjoint pair $L \dashv R$ with the following property. There is a collection $U$ of injectives in $\mathcal{B}$ such that every $X\in \mathcal{B}$ admits an admissible monomorphism $X\to Y$ where $Y$ is a product of objects in $U$, and $RX$ is injective in $\mathcal{A}$ for all $X \in U$. Then $L$ is exact.
\end{lem}
\begin{proof} Let $0 \to X \to Y \to Z \to 0$ be an admissible short exact sequence in $\mathcal{A}$. Then $LX \to LY$ is a morphism in $\mathcal{B}$ with cokernel $LZ$ since $L$ commutes with colimits as a left adjoint. Consider an admissible monomorphism $LX \to Q$ where $Q$ is a product of objects in $U$. This morphism corresponds, via adjointness, to a morphism $X \to RQ$ where $RQ$ is a product of injectives in $\mathcal{A}$ by assumption\linebreak ($R$ commutes with products as a right adjoint). By injectivity $X\to RQ$ factors through $X\to Y$ and so $LX \to Q$ factors through $LX \to LY$. It follows from the Obscure axiom, see \cite[Proposition 2.16]{Buehler}, that $LX \to LY$ is an admissible monomorphism and so 
$0\to LX \to LY \to LZ \to 0$ is an admissible short exact sequence. Hence $L$ is exact.
\end{proof}

\begin{coro} Let $U \subseteq \Ind \bar{\mathcal{C}}$ be a cogen-rigid subset.
\begin{itemize}
    \item[(1)] There is a cosilting subset of $\Ind \bar{\mathcal{C}}$ containing $U$.
    \item[(2)] The closure of $U$ in $\Ind \bar{\mathcal{C}}$ is cogen-rigid.
\end{itemize}
\end{coro}
\begin{proof} By \Cref{keylem} and \Cref{mai} the cogen-rigid subset $U$ is contained in the cosilting subset $U' = \Inj \overline{\cogen U}$ which is closed in $\Ind \bar{\mathcal{C}}$. Thus, the closure $\bar{U}$ of $U$ in $\Ind \bar{\mathcal{C}}$ is contained in $U'$ and fulfills $\cogen U = \cogen \bar{U} = \cogen U'$. It follows that $\bar{U}$ is cogen-rigid. 
\end{proof}

\section{The case of an Artin algebra}

Let $A$ be an Artin algebra (for example a finite dimensional algebra), $\Mod A$ the category of (left) $A$-modules, $\mod A \subseteq \Mod A$ the full subcategory of finite length modules, and set $\Ind A = \Ind \Mod A$. Torsion pairs in $\mod A$ can be classified via cosilting subsets $U \subseteq \Ind A$ by \Cref{mai}. One goal will be to characterize cosilting subsets in terms of the (infinite) Auslander-Reiten translation in the spirit of $\tau$-tilting theory \cite{air}.

\begin{df} \rm We follow the definition in \cite{Krause0} to deal with arbitrary $A$-modules. Let $\underline{\mathrm{Mod}}\, A$ and $\overline{\mathrm{Mod}}\,A$ be the projectively and injectively stable module categories, respectively. The \emph{Auslander-Reiten translation} is the functor $\tau \colon \underline{\mathrm{Mod}}\, A \to \overline{\mathrm{Mod}}\, A$ 
defined by the exact sequences
\begin{align*}
P_1 \longrightarrow P_0 \longrightarrow X \longrightarrow 0, \qquad 0 \longrightarrow \tau X \longrightarrow DA \otimes_A P_1 \to DA\otimes_A P_0
\end{align*}
for $X\in \Mod A$. Here $D$ is the standard duality and the first exact sequence is a minimal projective presentation of $X$. The definition of $\tau$ on morphisms is as usual given by lifting morphisms along projective covers. 
\end{df}

The main connection between the Auslander-Reiten translation and the Ziegler spectrum is the following.

\begin{prop} \label{homeo}\cite[Proposition 5.5 and Corollary 5.15]{Krause0} The Auslander-Reiten translation $\tau\colon \underline{\mathrm{Mod}}\, A \to \overline{\mathrm{Mod}}\, A$ is an equivalence and induces a homeomorphism $\Ind A \setminus \mathrm{Proj}\,A \to \Ind A \setminus \Inj A$ where $\mathrm{Proj}\,A$ and $\mathrm{Inj}\,A$ are the sets of indecomposable projective and injective $A$-modules, respectively.
\end{prop}

The inverse of $\tau$ is denoted by $\tau^{-}$ and is defined by the exact sequences 
\begin{align*}
0 \longrightarrow X \longrightarrow I_0 \longrightarrow I_1, \qquad \Hom_{A}(DA,I_0)  \to \Hom_{A}(DA, I_1) \longrightarrow \tau^- X \longrightarrow 0
\end{align*}
for $X\in \Mod A$ where the first exact sequence is a minimal injective presentation. We introduce the following notions analogous to classical $\tau$-tilting theory. 

\begin{df}\label{cotautilting}\rm For $X,Y\in \Mod A$ let $\Hom_A(X,Y)^\mathrm{fin} \subseteq \Hom_A(X,Y)$ be the subgroup of morphisms with finite length image.
\begin{itemize}
    \item[(1)] A subset $U\subseteq \Ind A$ is \emph{$\tau^-$-rigid} if $\Hom_A^\mathrm{fin}(\tau^- U, U) = 0$.
    \item[(2)] We call $(U, V)$ a \emph{$\tau ^-$-rigid} pair if $U \subseteq \Ind A$ is $\tau^-$-rigid and $V\subseteq \mathrm{Inj}\,A$ fulfills $\mathrm{Hom}_A(U,V) = 0$. There is a partial order on $\tau^-$-rigid pairs given by $(U, V) \leq (U', V')$ if $U\subseteq U'$ and $V\subseteq V'$.
    \item[(3)] A subset $U\subseteq \Ind A$ is \emph{support $\tau^-$-tilting} if there exists a maximal {$\tau^-$-rigid pair $(U, V)$}.
\end{itemize}
\end{df}

Next, we show that $\tau^-$-rigid subsets coincide with cogen-rigid subsets of $\Ind A$. The key ingredient is the following Auslander-Reiten formula. 

\begin{lem}\label{arf} \cite[Corollary]{Krause3a} For $X\in \mod A$ and $Y\in \Mod A$ there exists an isomorphism
\begin{align*}
    D\,\mathrm{Ext}^1_{A}(X, Y) \cong \overline{\mathrm{Hom}}_A(Y,\tau X) 
\end{align*}
functorial in $X$ and $Y$.
\end{lem}

\begin{thm}\label{rechar}  Let $U \subseteq \Ind A$. Then $U$ is cogen-rigid if and only if $U$ is $\tau^-$-rigid.
\end{thm}
\begin{proof} We will show that for all $X,Y \in \Mod A$ the following are equivalent:
\begin{align*}
    \mathrm{(1)}\quad \mathrm{Ext}_A^1 (\cogen X, Y) = 0, \qquad \mathrm{(2)}\quad \Hom_A^\mathrm{fin}(\tau^- Y, X) = 0.
\end{align*}
Here $\cogen X$ denotes all finite length submodules of arbitrary products of $X$. This implies the statement of the theorem by \Cref{length} as $\mod A$ is a length category. For $Z \in \mod A$ we obtain the Auslander-Reiten formula 
    \begin{align*}
    D\,\mathrm{Ext}^1_{A}(Z, Y) \cong \overline{\mathrm{Hom}}_A(Y,\tau Z) \cong \underline{\mathrm{Hom}}_A(\tau^- Y,Z)  
\end{align*}
 by \Cref{arf} and \Cref{homeo}.
 
(2)$\implies$(1) If $\mathrm{Ext}^1_A(\cogen X, Y) \neq 0$ then by the formula above, there exists a non-zero morphism $\tau^{-}Y \to Z$ for $Z\in \cogen X$. Further, there is a monomorphism $Z\to \prod X$ and the composition $ \tau^{-}Y \to Z \to \prod X$ is non-zero. But then composing with a suitable projection $\prod X \to X$ yields a non-zero morphism $\tau^- Y \to X$ with a finite length image contradicting (2).

(1)$\implies$(2) Let $\tau^- Y \to X$ be a morphism with finite length image $Z$. By (1) and the Auslander-Reiten formula, the canonical epimorphism $\tau^{-} Y \to Z $ factors through a projective $A$-module $P$ as $\tau^-{} Y \to P \to Z$. Without loss of generality we assume that $P \to Z $ is a projective cover. Now $P \to Z$ factors through the epimorphism $\tau^{-}Y \to Z$. In total $P \to Z$ factors as $P \to \tau^{-}Y \to P \to Z$ and it follows that $P \to \tau^{-}Y$ is a split monomorphism. But $\tau^- Y$ has no non-zero projective summands and so $P = 0$. Hence $\tau^- Y \to X$ is a zero morphism.
\end{proof}

\begin{coro}\label{taumain} There is a one-to-one correspondence 
\begin{equation*}
    \begin{Bmatrix}
        \text{torsion-free classes}\\
        \mathcal{F}\subseteq \mathrm{mod}\,A
    \end{Bmatrix} \longleftrightarrow 
    \begin{Bmatrix}
        \text{support $\tau^-$-tilting subsets}\\
        {U}\subseteq \Ind A
    \end{Bmatrix}
\end{equation*}
given by $\mathcal{F} \mapsto \Inj \bar{\mathcal{F}}$ and $U \mapsto \cogen U$. Moreover, support $\tau^-$-tilting subsets are closed.
\end{coro}
\begin{proof} By \Cref{rechar} a subset $U \subseteq \Ind A$ is support $\tau^-$-tilting if and only if $U$ is cosilting. Now apply \Cref{mai}.
\end{proof}

In the above classification we are possibly dealing with infinite length modules and the next goal is to reduce everything to finite length modules. The trade-off is that we have to consider ideals of the module category. A non-empty collection $\mathcal{I}$ of morphisms in $\mod A$ is an \emph{ideal} if for all $\varphi , \psi \in \mathcal{I}$ and arbitrary morphisms $\alpha, \beta$ in $\mod A$ we have $\varphi+ \psi \in \mathcal{I}$ and $\beta \varphi \alpha \in \mathcal{I}$ if the expressions are defined. We are interested in certain ideals introduced by Krause \cite{Krause}.

\begin{df}\rm An ideal $\mathcal{I}$ of $\mod A$ is \emph{fp-idempotent} if the full subcategory of finitely presented functors $F\colon \mod A \to \Ab$ with $F(\varphi) = 0$ for all $\varphi\in \mathcal{I}$ is closed under extensions. Recall that $F$ is \emph{finitely presented} if there exist $X,Y \in \mod A$ and a short exact sequence $\Hom_A(Y,-) \to \Hom_A(X,-) \to F \to 0$.
\end{df}

For $U \subseteq \Ind A$ let $\langle U\rangle$ be the ideal of morphisms in $\mod A$ that factor through a product of modules in $U$. The next result, due to Krause, is the main connection between fp-idempotent ideals and the Ziegler spectrum.

\begin{thm}\label{krau}\cite[Corollary 5.14]{Krause}  There is a one-to-one correspondence between closed sets $U \subseteq \Ind A$ and fp-idempotent ideals $\mathcal{I}$ of $\mod A$ given by $U \mapsto \langle U \rangle$.
\end{thm}

The following result by Prest offers a nice description of fp-idempotent ideals.

\begin{prop}\label{pres}\cite[Corollary 4.7]{Prest05} Let $U \subseteq \Ind A$ and $\bar{U}$ its closure. The ideal $\langle \bar{U}\rangle$ consists of all morphisms in $\mod A$ that factor through a finite direct sum of objects in $U$. In particular $\langle U\rangle = \langle \bar{U} \rangle$ is always fp-idempotent. \end{prop}

We will express $\tau^-$-rigidity of subsets $U \subseteq \Ind A$ in terms of the corresponding fp-idempotent ideal $\mathcal{I} = \langle U \rangle$. We define $\tau \mathcal{I} = \langle \tau U\rangle$ and $\tau^- \mathcal{I} = \langle \tau^- {U}\rangle$. These expressions are independent of the choice of $U$ by \Cref{pres} since $\tau$ is a homeomorphism, see \Cref{homeo}. Set 
\begin{align*}
\cogen \mathcal{I} &= \{X \in \mod A \mid \text{there is  a mono }X\to Y\text{ in }\mathcal{I}\} \\
&= \{X\in \mod A \mid \Hom_A(X,DA) = \mathcal{I}(X,DA)\}.
\end{align*}
The following notions are variants of \Cref{cotautilting} on the level of ideals.  
\begin{df}\label{id}\rm Let $\mathcal{I}$ be an ideal of $\mod A$.
\begin{itemize}
    \item[(1)] The ideal $\mathcal{I}$ is $\tau^-$-\emph{rigid} if $\mathcal{I}$ is fp-idempotent and $\mathcal{I}\circ \tau^- \mathcal{I} = 0$.
    \item[(2)] We call $(\mathcal{I}, V)$ a $\tau^-$-\emph{rigid} pair if $\mathcal{I}$ is $\tau^-$-rigid and $V\subseteq \Inj A$ fulfills $\mathcal{I}(X,V) = 0$ for all $X\in \mod A$. There is a partial order on $\tau^-$-rigid pairs given by $(\mathcal{I}, V) \leq (\mathcal{I}', V')$ if $\mathcal{I} \subseteq \mathcal{I}'$ and $V\subseteq V'$.
    \item[(3)] The ideal $\mathcal{I}$ is \emph{support $\tau^-$-tilting} if there is a maximal $\tau^-$-rigid pair $(\mathcal{I}, V)$.
\end{itemize}
\end{df}

\begin{prop}\label{core} Let $U \subseteq \Ind A, V \subseteq \Inj A$ and $\mathcal{I} = \langle U\rangle $. Then $(U,V)$ is $\tau^-$-rigid if and only if $(\mathcal{I}, V)$ is $\tau^-$-rigid and in this case $\cogen \mathcal{I} = \cogen U$ is a torsion-free class. 
\end{prop}

\begin{proof} First, we show that $U$ is $\tau^-$-rigid if and only if the ideal $\mathcal{I}$ is $\tau^-$-rigid.  By \Cref{pres} every morphism in $\mathcal{I} \circ \tau^- \mathcal{I}$ factors as 
\begin{align*}
    X\longrightarrow \bigoplus_{i=1}^n \tau^{-} Y_i \longrightarrow X' \longrightarrow \bigoplus_{j=1}^m  Y_j' \longrightarrow X''     
\end{align*}
with $X,X',X'' \in \mod A$ and $Y_i, Y_j' \in U$. If $U$ is $\tau^-$-rigid, then  the composition $\bigoplus \tau^- Y_i \to X' \to \bigoplus Y_j'$ is zero as $\Hom_A^\mathrm{fin}(\tau^- U, U) = 0$. Thus $\mathcal{I} \circ \tau^- \mathcal{I} = 0$.

For $X, Y \in U$ and a morphism $\tau^-Y \to X$ with finite length image $Z$ consider an epimorphism $\bigoplus A \to \tau^- Y$ and a monomorphism $X \to \prod DA$. If the ideal $\mathcal{I}$ is $\tau^-$-rigid then each component $A \to DA$ of the composition
\begin{align*}
    \bigoplus A \longrightarrow \tau^- Y \longrightarrow Z \longrightarrow X\longrightarrow \prod DA
\end{align*}
is zero as $\mathcal{I}\circ \tau^- \mathcal{I} = 0$. Thus $U$ is $\tau^-$-rigid.

Next, we show that $\Hom_A(U, Q) = 0$ if and only if $\mathcal{I}(-,Q) = 0$ for $Q\in \Inj A$. The only if part is trivial. Suppose that $\mathcal{I}(-,Q)= 0$. For $X\in U$ let $\bigoplus A \to X$ be an epimorphism. Then, for $X\to Q$ every component $A\to Q$ of the composition $\bigoplus A \to X \to Q$ is contained in $\mathcal{I}(A,Q) = 0$. Thus $X\to Q$ equals zero and $\Hom_A(U, Q) = 0$.

Lastly, we prove that $\cogen \mathcal{I}  = \cogen U$. Note that $U$ being $\tau^-$-rigid implies that $U$ is cogen-rigid by \Cref{rechar} and so $\cogen U$ is a torsion-free class, see \Cref{keylem}. The inclusion $\cogen \mathcal{I }\subseteq \cogen U$ is trivial. For $X$ in $\cogen U$ there is a monomorphism $X\to \prod Y$ where the product goes over some $Y\in U$. The injective hull $f\colon X\to Q$ factors through $X\to \prod Y$ and so $f\in \mathcal{I}$. Hence $X\in \cogen  \mathcal{I}$ and $\cogen \mathcal{I} = \cogen U$.
\end{proof}

\begin{coro}\label{mainideal} There is a one-to-one correspondence 
\begin{equation*}
    \begin{Bmatrix}
        \text{torsion-free classes}\\
        \mathcal{F}\subseteq \mod A
    \end{Bmatrix} \longleftrightarrow 
    \begin{Bmatrix}
        \text{support $\tau^-$-tilting ideals}\\
        \mathcal{I}\text{ of } \mod A
    \end{Bmatrix}
\end{equation*}
given by $\mathcal{F} \mapsto \langle \Inj \bar{\mathcal{F}} \rangle$ and $\mathcal{I} \mapsto \cogen \mathcal{I}$. 
\end{coro}
\begin{proof} The assignment $U \mapsto \langle U \rangle$ induces a one-to-one correspondence between support $\tau^-$-tilting subsets $U \subseteq \Ind A$ and support $\tau^-$-tilting ideals $\mathcal{I}$ of $\mod A$ such that $\cogen \mathcal{I} = \cogen U$ by \Cref{core}. Combining this correspondence with the one in 
\Cref{taumain} shows the desired result.
\end{proof}

\begin{exa}\label{kro}\rm Consider the Kronecker algebra
\begin{equation*}
    \begin{tikzcd}
    A =k\,(\bullet \arrow[r, shift left] \arrow[r, shift right] & \bullet)
\end{tikzcd}
\end{equation*}
over an algebraically closed field $k$. The indecomposable modules in $\mod A$ can be divided into three parts: The preprojective modules $\mathcal{P} = \{P_1, P_2, \dots \}$, the preinjective modules $\mathcal{Q} = \{Q_1, Q_2, \dots \}$ and the regular modules  $\mathcal{R}$, which further divide into tubes $\mathcal{R}^\lambda = \{R^\lambda_1, R^\lambda_2, \dots\}$ with $\lambda\in k\cup \{\infty \}$. The Auslander-Reiten quiver of $\mod A$ can be visualized as follows.
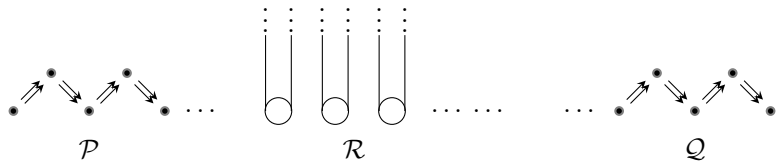
\tikzstyle{place}=[circle,draw=black!50,fill=black!100,thick,
inner sep=0pt,minimum size=1mm]
\begin{align*}
\begin{tikzpicture}
\draw (1, -0.5) node{$\mathcal{P}$};
\draw (9, -0.5) node{$\mathcal{Q}$};
\draw (4.5, -0.5) node{$\mathcal{R}$};
\draw (0,0) node[place]{};
\draw [-stealth](0.1,0.15) -- (0.35,0.4);
\draw [-stealth](0.15,0.1) -- (0.4,0.35);
\draw (0.5,0.5) node[place]{};
\draw [-stealth](0.65,0.4) -- (0.9,0.15);
\draw [-stealth](0.6,0.35) -- (0.85,0.1);
\draw (1,0) node[place]{};
\draw [-stealth](1.1,0.15) -- (1.35,0.4);
\draw [-stealth](1.15,0.1) -- (1.4,0.35);
\draw (1.5,0.5) node[place]{};
\draw [-stealth](1.65,0.4) -- (1.9,0.15);
\draw [-stealth](1.6,0.35) -- (1.85,0.1);
\draw (2,0) node[place]{};
\draw (2.5,0) node{$\dots$};
\draw (10,0) node[place]{};
\draw [-stealth](9.1,0.15) -- (9.35,0.4);
\draw [-stealth](9.15,0.1) -- (9.4,0.35);
\draw (9.5,0.5) node[place]{};
\draw [-stealth](9.65,0.4) -- (9.9,0.15);
\draw [-stealth](9.6,0.35) -- (9.85,0.1);
\draw (9,0) node[place]{};
\draw [-stealth](8.65,0.4) -- (8.9,0.15);
\draw [-stealth](8.6,0.35) -- (8.85,0.1);
\draw (8.5,0.5) node[place]{};
\draw [-stealth](8.1,0.15) -- (8.35,0.4);
\draw [-stealth](8.15,0.1) -- (8.4,0.35);
\draw (8,0) node[place]{};
\draw (7.5,0) node{$\dots$};
\draw (3.5,0) circle (5pt);
\draw (3.33,0)--(3.33,1);
\draw (3.33,1.3) node{$\vdots$};
\draw (3.67,0)--(3.67,1);
\draw (3.67,1.3) node{$\vdots$};
\draw (4.25,0) circle (5pt);
\draw (4.08,0)--(4.08,1);
\draw (4.08,1.3) node{$\vdots$};
\draw (4.42,0)--(4.42,1);
\draw (4.42,1.3) node{$\vdots$};
\draw (5,0) circle (5pt);
\draw (4.83,0)--(4.83,1);
\draw (4.83,1.3) node{$\vdots$};
\draw (5.17,0)--(5.17,1);
\draw (5.17,1.3) node{$\vdots$};
\draw (5.75,0) node{$\dots$};
\draw (6.3,0) node{$\dots$};
\end{tikzpicture}
\end{align*}

For all $\lambda$ the Pr\"ufer module $R_\infty^\lambda$ equals a filtered colimit of monomorphisms in $\mathcal{R}^\lambda$ and the adic module $\hat{R}^\lambda$ is given by an inverse limit of epimorphisms in $\mathcal{R}^\lambda$.\linebreak The Ziegler spectrum $\Ind A$ consists of all indecomposable modules in $\mod A$, the Pr\"ufer modules $R_\infty^\lambda$, the adic modules $\hat{R}^\lambda$, and the unique generic module $G$, see for example \cite[Theorem 14.2.15]{Krause4}. Moreover, a subset $U \subseteq \Ind A$ is closed if and only if the following holds. 
\begin{itemize}
    \item[(1)] If $U$ contains infinitely many modules in $\mathcal{R}_\lambda$ then $U$ contains $R_\infty^\lambda$ and $\hat{R}^\lambda.$
    \item[(2)] If $U$ contains infinitely many modules in $\mathcal{P}$ then $U$ contains $\hat{R}^\lambda$ for all $\lambda$.
    \item[(3)] If $U$ contains infinitely many modules in $\mathcal{Q}$ then $U$ contains ${R}^\lambda_\infty$ for all $\lambda$.
    \item[(4)] If $U$ contains infinitely many modules or a module of infinite length, then $U$ contains $G$. 
\end{itemize}
From the topology of the Ziegler spectrum and the behavior of $\tau$ on finite length modules, it follows that $\tau X \cong X$ for every infinite length module $X\in \Ind A$ since $\tau \colon \Ind A \setminus \{P_1, P_2\} \to \Ind A \setminus \{Q_1, Q_2\}$ is a homeomorphism, see \Cref{homeo}. This will be used to determine all $\tau^-$-rigid subsets $U \subseteq \Ind A$ that consist of infinite length modules only. For our calculations, we apply fp-idempotent ideals and make use of the fact that there are only zero morphisms from the right to the left in the displayed Auslander-Reiten quiver.

The fp-idempotent ideal $\langle G\rangle$ is generated by all morphisms from  $\mathcal{P}$ to $\mathcal{Q}$.  For $\lambda \in k\cup \{\infty\}$ the fp-idempotent ideal $\langle R^\lambda_\infty \rangle$ is generated by all morphisms from $\mathcal{R}^\lambda$ to $\mathcal{Q}$, and $\langle \hat{R}^\lambda\rangle$ is generated by all morphisms  from $\mathcal{P}$ to  $\mathcal{R}^\lambda$. The only non-zero compositions, using these ideals, are given by $\langle R_\infty ^\lambda \rangle \circ \langle \hat{R}^\lambda \rangle \neq 0$. Since $\tau X \cong X$ for infinite length modules $X\in \Ind A$ it follows from \Cref{core} that a subset $U \subseteq \Ind A \setminus \mod A$ is $\tau^-$-rigid if and only if $R_\infty^\lambda \in U$ implies $\hat{R}^\lambda \notin U$. Assume that  $U$ is such a $\tau^-$-rigid non-empty set. Then the torsion-free class $\mathcal{F}= \cogen U$ associated to $U$ is given by the additive closure of $\mathcal{P}$ as well as all $ \mathcal{R}^\lambda$ with ${R}_\infty^\lambda \in U$. Now $U$ is support $\tau^-$-tilting if and only if 
    \begin{align*}
        U = \{R_\infty^\lambda, \hat{R}^\mu , G\mid \lambda\in S, \mu \in T\}
    \end{align*}
    with $S\dot{\cup } T = k \cup \{\infty\}$. The corresponding support $\tau^-$-tilting ideal $\langle U \rangle$ is generated by all morphisms from $\mathcal{R}^\lambda$ to $\mathcal{Q}$, and $\mathcal{P}$ to $\mathcal{R}^\mu$ with $\lambda \in S$ and $\mu \in T$.
\end{exa}

\section{The dual perspective}

Let $A$ be an Artin algebra. The aim of this section is to present dual counterparts of the notions and results developed in the previous section. The difficulty lies in dealing with infinite length $A$-modules for which duality is not well behaved. We overcome this by working on the level of ideals first. The following definition is dual to \Cref{id}.

\begin{df}\label{id2}\rm Let $\mathcal{I}$ be an ideal of $\mod A$.
\begin{itemize}
    \item[(1)] The ideal $\mathcal{I}$ is $\tau$-\emph{rigid} if $\mathcal{I}$ is fp-idempotent and $\tau \mathcal{I}\circ \ \mathcal{I} = 0$.
    \item[(2)] We call $(\mathcal{I}, V)$ a $\tau$-\emph{rigid} pair if $\mathcal{I}$ is $\tau$-rigid and $V\subseteq \mathrm{Proj}\, A$ fulfills\linebreak $\mathcal{I}(V,X) = 0$ for all $X\in \mod A$. There is a partial order on $\tau$-rigid pairs given by $(\mathcal{I}, V) \leq (\mathcal{I}', V')$ if $\mathcal{I} \subseteq \mathcal{I}'$ and $V\subseteq V'$.
    \item[(3)] The ideal $\mathcal{I}$ is \emph{support $\tau$-tilting} if there is a maximal $\tau$-rigid pair $(\mathcal{I}, V)$.
\end{itemize}
\end{df}

One issue is that these notions still make reference to infinite length $A$-modules, namely through the expression $\tau\mathcal{I}$. We will show that $\tau\mathcal{I}$ can instead be described using the Auslander-Reiten translation on morphisms. For an ideal $\mathcal{I}$ of $\mathrm{mod}\,A$ let $\underline{\mathcal{I}}$ and $\overline{\mathcal{I}}$ be the induced ideals in $\underline{\mathrm{mod}}\,A$ and $\overline{\mathrm{mod}}\,A$ respectively. Applying $\tau$ and $\tau^-$ on morphisms naturally yields an ideal $\tau \underline{\mathcal{I}}$ of $\overline{\mathrm{mod}}\,A$ and an ideal $\tau^- \overline{ \mathcal{I}}$ of $\underline{\mathrm{mod}}\,A$.
The following relates $\tau \mathcal{I}$ and $\tau \underline{\mathcal{I}}$ as well as $\tau^- \mathcal{I}$ and $\tau^- \overline{\mathcal{I}}$.

\begin{lem}\label{finite} Let $\mathcal{I}$ be an fp-idempotent ideal of $\mod A$.
\begin{itemize}
    \item[(1)] The ideal $\tau \mathcal{I}$ equals the unique fp-idempotent ideal $\mathcal{J}$ of $\mod A$ containing no $1_Q$ with $Q\in \Inj A$ such that $\tau \underline{\mathcal{I}} = \overline{\mathcal{J}}$.
    \item[(2)] The ideal $\tau^- \mathcal{I}$ equals the unique fp-idempotent ideal $\mathcal{J}$ of $\mod A$ containing no $1_P$ with $P\in \mathrm{Proj}\,A$ such that $\tau^- \overline{\mathcal{I}} = \underline{\mathcal{J}}$.
\end{itemize}
\end{lem}
\begin{proof}  We only show (1) as (2) is similar. Let $\mathcal{I} = \langle U \rangle$ for a closed set $U \subseteq \Ind A$. Since the Auslander-Reiten translation $\tau \colon \underline{\mathrm{Mod}}\, A \to \overline{\mathrm{Mod}}\, A$ is an equivalence by \Cref{homeo} it follows that $\tau \underline{\mathcal{I}} = \overline{\tau \mathcal{I}}$. Moreover $Q\notin \tau U$ implies that $1_Q \notin \tau \mathcal{I}$ for $Q\in \Inj A$. Let $\mathcal{J} = \langle U'\rangle$ have the properties stated with $U'\subseteq \Ind A$ closed. Then $\overline{\mathcal{J}} = \overline{\tau \mathcal{I}}$ implies that $\mathcal{J} + \langle \Inj A \rangle = \tau \mathcal{I}+\langle \Inj A \rangle$ and so $U' \cup \Inj A = \tau U \cup \Inj A$ \mbox{by \Cref{krau}}. By assumption $U' \cap \Inj A = \emptyset$ and thus $U ' = \tau {U}$. We conclude that $\mathcal{J} = \tau \mathcal{I}$.   
\end{proof}

For an ideal $\mathcal{I}$ and an additive subcategory $\mathcal{C}$ of $\mod A$ the duality $D$ naturally induces an ideal $D \mathcal{I}$ and an additive subcategory $D\mathcal{C}$ of $\mod A ^\op$. Let $\gen \mathcal{I}$ be the collection of all $X\in \mod A$ that admit an epimorphism $Y \to X $ in $\mathcal{I}$.

\begin{prop}\label{duality} Let $\mathcal{I}$ be an ideal of $\mod A$ and $V\subseteq \mathrm{Proj}\, A$. Then $(\mathcal{I}, V)$ is $\tau$-rigid if and only if $(D \mathcal{I}, DV)$ is $\tau^-$-rigid. Moreover $\gen \mathcal{I} = D\,\cogen D\mathcal{I}$. 
\end{prop}
\begin{proof} By  \cite[Remark 4.6 (1)]{self2} the ideal $\mathcal{I}$ is fp-idempotent if and only if $D\mathcal{I}$ is fp-idempotent. Thus, the definition of being \mbox{$\tau$-rigid} is completely dual to being \mbox{$\tau^-$-rigid} and may only be expressed in terms of finite length modules by \Cref{finite}. It follows that $(\mathcal{I}, V)$ is $\tau$-rigid if and only if $(D \mathcal{I}, DV)$ is $\tau^-$-rigid. Moreover, any epimorphism $X\to Y$ in $\mathcal{I}$ corresponds to a monomorphism $DY \to DX$ in $D\mathcal{I}$ and so $\gen \mathcal{I} = D \, \cogen D\mathcal{I}$.
\end{proof}

\begin{coro}\label{mainideal2} There is a one-to-one correspondence 
\begin{equation*}
    \begin{Bmatrix}
        \text{torsion classes}\\
        \mathcal{T}\subseteq \mod A
    \end{Bmatrix} \longleftrightarrow 
    \begin{Bmatrix}
        \text{support $\tau$-tilting ideals}\\
        \mathcal{I}\text{ of } \mod A
    \end{Bmatrix}
\end{equation*}
given by $\mathcal{I} \mapsto \gen \mathcal{I}$.
\end{coro}
\begin{proof} The result follows from duality, see \Cref{duality}, by \Cref{mainideal}.
\end{proof}

We also want to establish a dual result of \Cref{taumain} for the module category of an Artin algebra. This requires the following definition.

\begin{df}\label{tautilting}\rm \phantom{X}
\begin{itemize}
    \item[(1)] A subset $U\subseteq \Ind A$ is \emph{$\tau$-rigid} if $\Hom_A^\mathrm{fin}(U, \tau U) = 0$.
    \item[(2)] We call $(U, V)$ a \emph{$\tau$-rigid} pair if $U \subseteq \Ind A$ is $\tau$-rigid and $V\subseteq \mathrm{Proj}\,A$ fulfills $\mathrm{Hom}_A(V,U) = 0$. There is a partial order on $\tau$-rigid pairs given by $(U, V) \leq (U', V')$ if $U\subseteq U'$ and $V\subseteq V'$.
    \item[(3)] A subset $U\subseteq \Ind A$ is \emph{support $\tau$-tilting} if there exists a maximal {$\tau$-rigid pair $(U, V)$}.
\end{itemize}
\end{df}

For $U \subseteq \Ind A$ let $\gen U $ be the collection of all $X\in \mod A$ that admit an epimorphism $\bigoplus_{i=1}^n Y_i\to X$ with $Y_i\in U$. The following is precisely the dual version of \Cref{core}.

\begin{prop}\label{core2} Let $U \subseteq \Ind A, V \subseteq \mathrm{Proj}\,A$ and $\mathcal{I} = \langle U\rangle $. Then $(U,V)$ is $\tau$-rigid if and only if $(\mathcal{I}, V)$ is $\tau$-rigid and in this case $\gen \mathcal{I} = \gen U$ is a torsion class. 
\end{prop}
\begin{proof} The fact that $(U,V)$ is  $\tau$-rigid if and only if $(\mathcal{I},V)$ is $\tau$-rigid can be shown as in the proof of \Cref{core}. Moreover $\gen \mathcal{I} = D\,\cogen D\mathcal{I} $ is a torsion class in this case by duality, see \Cref{duality}.
\end{proof}

\begin{coro} \label{main2} There is a one-to-one correspondence 
\begin{equation*}
    \begin{Bmatrix}
        \text{torsion classes}\\
        \mathcal{T}\subseteq \mod A
    \end{Bmatrix} \longleftrightarrow 
    \begin{Bmatrix}
        \text{support $\tau$-tilting subsets}\\
        {U}\subseteq \Ind A
    \end{Bmatrix}
\end{equation*}
given by  $U \mapsto \gen U$. Moreover, support $\tau$-tilting subsets are closed.
\end{coro}
\begin{proof} By \Cref{core2} it follows that support $\tau$-tilting subsets $U \subseteq \Ind A$ are one to one with support $\tau$-tilting ideals $\mathcal{I}$ and $\gen U = \gen \mathcal{I}$. Now, the desired correspondence follows from \Cref{mainideal2}. The equality $\langle U\rangle = \langle \bar{U} \rangle$ for $U \subseteq \Ind A$, see \Cref{pres}, shows that support $\tau$-tilting subsets are closed.
\end{proof}

The above is a dual variant of the correspondence in \Cref{taumain}. Note that for torsion-free classes $\mathcal{F}\subseteq \mod A$ the associated support $\tau^-$-tilting subset may be recovered via $\mathcal{F}\mapsto \Inj \bar{\mathcal{F}}$. We do not know how to directly recover the support $\tau$-tilting subset corresponding to a torsion class. However, this will be achieved indirectly by investigating the relation between the support $\tau$-tilting subset and the support $\tau^-$-tilting subset corresponding to a torsion pair.
Recall that every indecomposable projective $P$ corresponds to a simple $S$ which corresponds to an injective $Q$. We set $P_* = Q$ and $Q_* = P$.

\begin{prop}\label{rel} Let $(\mathcal{T}, \mathcal{F})$ be a torsion pair in $\mod A$. Moreover, let $(U,V)$ be the maximal $\tau$-rigid pair corresponding to $\mathcal{T}$ and $(U',V')$ the maximal $\tau^-$-rigid pair corresponding to $\mathcal{F}$ with $U,U'\subseteq \Ind A$. Then $U' = \tau U \cup V_*$ and $U = \tau^-U'\cup V'_*$.
\end{prop}
\begin{proof} First, we show that $\tau U \cup V_* \subseteq  U' = \Inj \bar{\mathcal{F}}$. The equality $\Hom_A(V, U) =0$ implies $\Hom_A( U, V_*) =0$ and so $\Hom_A(\mathcal{T}, V_*) =\Hom_A(\gen U , V_*) = 0$. Thus $\mathcal{F}$ contains $V_*$ and even $V_* \subseteq \Inj \bar{\mathcal{F}} $ since $V_*$ consists of injective $A$-modules. Moreover $\Hom_A^\mathrm{fin}(U,\tau U) = 0$ shows that $\Hom_A(\mathcal{T}, \cogen \tau U) = \Hom_A(\gen U, \cogen \tau U) = 0$ and so $\cogen \tau U \subseteq \mathcal{F}$. Hence $\tau U\subseteq \bar{\mathcal{F}}$. Applying the Auslander-Reiten formula in \cite[Corollary]{Krause3a} yields
\begin{align*}
    D\Ext_A^1(Y, \tau X ) \cong \overline{\mathrm{Hom}}_A(\tau X,  \tau Y) \cong \underline{\mathrm{Hom}}_A(X,Y) = 0
\end{align*}
for $X\in U$ and $Y\in \mathcal{F}$ since $\Hom_A(\gen U , \mathcal{F}) = \Hom_A(\mathcal{T},\mathcal{F}) = 0$. It follows that $\Ext_A^1(\mathcal{F}, \tau U) = 0$ and so $\tau U \subseteq \Inj \bar{\mathcal{F}}$.

The inclusion $\tau U \cup V_* \subseteq U'$ implies $\tau \mathcal{I} + \langle V_*\rangle \subseteq \mathcal{I}'$ where $\mathcal{I}$ is the support $\tau$-tilting ideal and $\mathcal{I}'$ the support $\tau^-$-tilting ideal corresponding to $\mathcal{T}$ and $\mathcal{F}$, respectively. On the level of ideals we can apply duality, see \Cref{duality}, to deduce that $\tau^- \mathcal{I}'+\langle V'_* \rangle \subseteq \mathcal{I}$. On the level of the Ziegler spectrum we obtain $\tau^- U'\cup V'_* \subseteq U$. Thus
\begin{align*}
    U'\supseteq \tau U \cup V_* \supseteq \tau (\tau^- U' \cup V'_*) \cup V_*= (U'\setminus \mathrm{Inj}\,A) \cup V_*. 
\end{align*}
Clearly $U' \cap \mathrm{Inj}\,A = V_*$ consists of all $Q\in \Inj A$ that fulfill $\Hom_A(\mathcal{T}, Q) = 0$ or equivalently $\Hom_A(U, Q ) = 0$. It follows that the above inclusions are equalities and so $U' = \tau U \cup V_*$. Similarly $U = \tau^- U'\cup V'_*$.
\end{proof}

\begin{exa}\label{kro2}\rm Let $A$ be the Kronecker algebra as in \Cref{kro}. We already determined the support $\tau^-$-tilting subsets $U\subseteq \Ind A$ that consist only of infinite length modules as
    \begin{align*}
        U = \{R_\infty^\lambda, \hat{R}^\mu , G\mid \lambda\in S, \mu \in T\}
    \end{align*}
    for $S\dot{\cup } T = k \cup \{\infty\}$. The corresponding torsion-free class $\mathcal{F}$ is given by the additive closure of $\mathcal{P}$ and all $\mathcal{R}^\lambda$ with $\lambda \in S$. Let $\mathcal{T}$ be the associated torsion class. By \Cref{rel} we know that $\tau U$ is the support $\tau$-tilting subset corresponding to $\mathcal{T}$. Since the infinite length $A$-modules are $\tau$-invariant here, it follows that $\mathcal{T} = \gen U$ which is the additive closure of $\mathcal{Q}$ and all $\mathcal{R}^\mu$ with $\mu \in T$.
\end{exa}

\section{Generic modules and bricks}

We will apply the developed theory to study special points inside the Ziegler spectrum of an Artin algebra $A$. Recall that an $A$-module $X$ is a \emph{brick} if $\mathrm{End}_A(X)$ is a division ring, and $X$ is \emph{generic} if $X$ is indecomposable, of infinite length, and has finite length over $\mathrm{End}_A(X)$. Generic modules are closed points inside $\Ind A$, see \cite[Proposition 6.17]{Krause}. We are particularly interested in generic bricks because of the following conjecture formulated by Mousavand and Paquette.

 \begin{conj}\label{strong}\cite[Conjecture 1.1]{MP} For a finite dimensional algebra $A$ over an algebraically closed field, the following are equivalent.
 \begin{itemize}
     \item[(1)] There are infinitely many non-isomorphic finite dimensional bricks.
     \item[(2)] There are infinitely many non-isomorphic finite dimensional bricks of the same dimension. 
     \item[(3)] There exists a generic brick.
 \end{itemize}
 \end{conj}

The equivalence between (2) and (3) of the conjecture has been shown for tame algebras by Bautista, P\' erez and Salmer\'on \cite{BPS}. We provide an alternative proof for tame algebras using infinite $\tau$-tilting theory. Moreover, we will prove the full conjecture under the assumption that the Krull-Gabriel dimension of $A$ is defined. A conjecture of Prest states that the Krull-Gabriel dimension of $A$ is defined if and only if the algebra $A$ is domestic 
\cite[Conjecture 9.1.15]{Prest}, which is known for several classes of algebras,  see \cite[Section 1]{Pastu} for a detailed discussion.

For now, we are still working over an arbitrary Artin algebra $A$. We say that  $X\in \Ind A$ is \emph{$\tau^-$-rigid} if $\{X\} \subseteq \Ind A$ is $\tau^-$-rigid. These modules coincide with the notion of grains in \cite{HLS}. The $\tau^-$-perspective is preferred to the dual perspective here, because it admits a more natural interpretation of the support $\tau^-$-tilting subset corresponding to a torsion-free class.
%\begin{coro}\label{oneimpl} Let $k$ be an algebraically closed field and $A$ a finite dimensional $k$-algebra. If $A$ is tame, then every generic brick is $\tau$-stable and rigid.
%\end{coro}
%\begin{proof} Over a tame algebra, every generic $A$-module $X$ is the accumulation point of infinitely many $\tau$-stable finite length $A$-modules in $ \Ind A$ and so $X$ is $\tau$-stable by \Cref{homeo}. This was already observed in \cite[Corollary 5.16]{Krause0}. Now if $X$ is also a brick, then every morphism $X\to X$ with a finite length image must be a non-isomorphism and hence zero. Thus $\{X\}\subseteq \Ind A$ is rigid by \Cref{rechar}. 
%\end{proof}

%We show that the converse of the above corollary holds. That is, every generic rigid module is a brick over a tame algebra. Before, we need the following lemma that holds for arbitrary Artin algebras.

\begin{lem}\label{genericbrick} Let $X$ be a generic $\tau^-$-rigid $A$-module. We also view $X$ as a module over $\mathrm{End}_A(X)$ and denote by $S = \mathrm{soc}_{\mathrm{End}_A(X)}X$ the socle of $X$ over $\mathrm{End}_A(X)$.
\begin{itemize}
    \item[(1)] The $A$-module $S$ is a generic brick and there is a short exact sequence
    \begin{align*}
        0 \longrightarrow S \longrightarrow X \longrightarrow Y \longrightarrow 0
    \end{align*}
    inside $\bar{\mathcal{F}}$ for the torsion-free class $\mathcal{F} = \cogen X$.
    \item[(2)] Every torsion-free class $\mathcal{G} \subseteq \Mod A$ containing $S$ also contains $X$. 
\end{itemize}
\end{lem}
\begin{proof} (1) Consider the canonical map $X\to \prod X$ of $A$-modules where the product goes over all radical morphisms $X\to X$ and let $Y$ be its image. There is a short exact sequence $0 \to S \to X \to Y \to 0$ and $Y$ is contained inside the torsion-free class $\bar{\mathcal{F}}\subseteq \Mod A$ by \Cref{keylem}. Note that this short exact sequence is exactly the reject sequence in \cite[Definition 6.4]{HLS}.

Because $X$ is injective in $\bar{\mathcal{F}}$ it follows that for all $f\colon S \to S$ there is $g\colon X \to X$ making the diagram 
\begin{equation*}
    \begin{tikzcd}
        0 \arrow[r] & S \arrow[r] \arrow[d, "f", swap] & X \arrow[r] \arrow[d, "g"] & Y \arrow[d] \arrow[r] & 0\\
                0 \arrow[r] & S \arrow[r] & X \arrow[r] & Y \arrow[r] & 0
    \end{tikzcd}
\end{equation*}
commute. Conversely, every $g\colon X \to X$ gives rise to such a commutative diagram. In particular, the length of $S$ over its endomorphism ring coincides with the length of $S$ over $\mathrm{End}_A(X)$ which is bounded by the length of $X$ over $\mathrm{End}_A(X)$. Now, if $g$ is a radical morphism, then $f = 0$ and otherwise both $g$ and $f$ are isomorphisms. It follows that $\mathrm{End}_A(S)$ coincides with  $D= \mathrm{End}_A(X)/ \mathrm{rad}\,\mathrm{End}_A(X)$ which is a division algebra of infinite length over the center of $A$ since $X$ is generic. Thus $S$ is a generic brick.

(2) Consider the socle sequence
\begin{align*}
     S = S_1 \subsetneq  S_2\subsetneq \dots \subsetneq S_n = X
\end{align*}
where $S_i$ equals the intersection of kernels of morphisms in $M_i = (\mathrm{rad}\, \mathrm{End}_A(X))^i$ and $n$ is the smallest natural number with the property $M_n = 0$. Such $n$ exists by the Nakayama lemma as $X$ is of finite length over its endomorphism ring. For $i>1$ there is a natural monomorphism 
\begin{align*}
S_i/S_{i-1} \longrightarrow \prod S,\quad  x+S_{i-1}\mapsto (\varphi(x))_\varphi    
\end{align*}
where the product goes over all morphisms $\varphi \in M_{i-1}$. Hence, every torsion-free class $\mathcal{G} \subseteq \Mod A$ containing $S$ also contains $X$ since $\mathcal{G}$ is closed under products, submodules and extensions.
\end{proof}

\begin{prop}\label{inj} There is, up to isomorphism, an injective assignment
\begin{equation*}
    \begin{Bmatrix}
        \text{generic $\tau^-$-rigid $A$-modules $X$}
    \end{Bmatrix}\longrightarrow 
    \begin{Bmatrix}
        \text{generic bricks }$S$\text{ over }$A$
    \end{Bmatrix}
\end{equation*}
given by $X \mapsto \mathrm{soc}_{\mathrm{End}_A(X)} X$.
\end{prop}
\begin{proof} Let $X,X'$ be generic $\tau^-$-rigid $A$-modules such that $S = \mathrm{soc}_{\mathrm{End}_A(X)} X$ is isomorphic to $S' = \mathrm{soc}_{\mathrm{End}_A(X')} X'$.  The modules $S$ and $S'$ are generic bricks over $A$ by \Cref{genericbrick} (1) and there are short exact sequences
\begin{align*}
    0 \longrightarrow S \longrightarrow X \longrightarrow Y \longrightarrow 0, \qquad 0 \longrightarrow S' \longrightarrow X' \longrightarrow Y' \longrightarrow 0 
\end{align*}
where the first one is inside $\overline{\cogen X}$ and the second inside $\overline{\cogen X'}$. Since $S\cong S'$ is contained in both torsion-free classes, it follows that $\mathcal{G}:=\overline{\cogen X'} = \overline{\cogen X}$ by \Cref{genericbrick} (2). Now $X$ and $X'$ are injective in $\mathcal{G}$ by \Cref{keylem}. Thus, the isomorphism $ S\cong S'$ induces a commutative diagram
\begin{equation*}
\begin{tikzcd}
    0 \arrow[r] & S \arrow[d, "\rotatebox{90}{$\sim$}", swap] \arrow[r] & X \arrow[d, "\alpha"] \arrow[r] & Y \arrow[d] \arrow[r] & 0 \\
    0 \arrow[r] & S' \arrow[d, "\rotatebox{90}{$\sim$}", swap] \arrow[r] & X' \arrow[d, "\beta"] \arrow[r] & Y' \arrow[d] \arrow[r] & 0\\
    0 \arrow[r] & S \arrow[r] & X \arrow[r] & Y \arrow[r] & 0.
\end{tikzcd}
\end{equation*} In particular $(\beta \alpha)^n \neq 0$ for all $n\in \mathbb{N}$. Because $X$ has finite length over $\mathrm{End}_A(X)$ there is $m \in \mathbb{N}$ with $\gamma^m = 0$ for every radical morphism $\gamma\colon X\to X$ by the Nakayama lemma. It follows that $\beta \alpha$ is an isomorphism and so $\alpha$ is a split monomorphism. Since $X,X'$ are indecomposable, we conclude that $X\cong X'$.
\end{proof}

The previous proposition has been independently proven in the work of Angeleri H\"ugel, Laking, and Pfeifer \cite{ALP}. We show how to obtain a generic $\tau^-$-rigid module from a collection of finite length modules, based on the following lemma.

\begin{lem}\label{comp} Let $U \subseteq \Ind A$ be a collection of finite length modules. The ideal $\mathcal{I}= \langle \bar{U} \setminus U\rangle$ consists of all morphisms in $\mod A$ that factor through a finite direct sum of modules in $U \setminus V$ for all finite subsets $V\subseteq U$.
\end{lem}

\begin{proof}
    For every finite subset $V \subseteq U$ consider the fp-idempotent ideal $\mathcal{I}_V = \langle U \setminus V\rangle$. By  \Cref{pres} the ideal $\mathcal{I}_V$ corresponds, under \Cref{krau}, to the closure of $U\setminus V$ in $\Ind A$ which equals $\bar{U} \setminus V$ since the finite length modules are open points in $\Ind A$, see \cite[Example 3.7]{Krause}. Let $\mathcal{J} = \bigcap \mathcal{I}_V$ where the intersection goes over all finite subsets $V \subseteq U$. Then $\mathcal{J}$ is fp-idempotent as the intersection is directed, see \cite[Remark 5.4]{self2}. Now, under \Cref{krau}, the ideal $\mathcal{J}$ corresponds to the closed set $\bigcap \bar{U} \setminus V = \bar{U} \setminus U$. Hence $\mathcal{I} = \mathcal{J}$ consists of all morphisms in $\mod A$ that factor through a finite direct sum of modules in $U\setminus V$ for all finite subsets $V \subseteq U$.
\end{proof}

\begin{prop}\label{generate} Let $U \subseteq \Ind A$ be an infinite collection of finite length modules such that
\begin{itemize}
    \item[(1)] for all but finitely many $X\in U$ we have $\Hom_A(\tau^- Y , X ) = 0$ for all but finitely many $Y \in U$, or
    \item[(2)] for all but finitely many $Y\in U$ we have $\Hom_A(\tau^- Y , X) = 0$ for all but finitely many $X \in U$.
\end{itemize}
Then $\bar{U} \setminus U \subseteq \Ind A$ is a non-empty $\tau^-$-rigid set and the corresponding torsion-free class $\mathcal{F} = \cogen \bar{U}\setminus U$ is given by 
\begin{align*}
    \mathcal{F} = \bigcap_{\substack{V\subseteq U\\ \text{finite}}} \cogen U\setminus V.
\end{align*}
Moreover, if the length of $X$ over $\mathrm{End}_A(X)$ is bounded for $X\in U$, then $\bar{U}\setminus U$ contains a generic $\tau^-$-rigid module.  
\end{prop}
\begin{proof} We only prove the case (1) as (2) is similar. By \Cref{core} the set $\bar{U}\setminus U$ is $\tau^-$-rigid if and only if $\mathcal{I} = \langle \bar{U} \setminus U \rangle$ is $\tau^-$-rigid. By \Cref{homeo} the ideal $\tau^- \mathcal{I} $ agrees with $\langle \tau^- \bar{U} \setminus \tau^- U \rangle$. With \Cref{comp} it follows that $\mathcal{I}$ consists of all morphisms that factor through a finite direct sum of modules in $U \setminus V$ for all finite subsets $V\subseteq U$, and $\tau^- \mathcal{I}$ consists of all morphisms that factor through a finite direct sum of modules in $\tau^- (U \setminus  W)$ for all finite subsets $W \subseteq U$. Consider a composition $\varphi \psi$ with $\varphi\in \mathcal{I}$ and $\psi \in \tau^- \mathcal{I}$. By assumption, there is a finite subset $V \subseteq U$ such that for all $X \in U \setminus V$ we have $\Hom_{A}(\tau^- Y, X) = 0$ for all but finitely many $Y \in U$. Now $\varphi$ factors through a finite direct sum $\bigoplus_{i=1}^n X_i$ with $X_i \in U \setminus V$. For each $i$ let $W_i$ be the finite set of all modules $Y\in  U$ with $\Hom_{A}(\tau^- Y, X_i)\neq 0$ and set $W = \bigcup_{i=1}^n W_i$. Then $\psi$ factors through a finite direct sum $\bigoplus_{j=1}^m \tau^- Y_j$ with $Y_j \in U \setminus W$. By construction $\Hom_A (\tau^- Y_j, X_i) = 0$ for all $i,j$ and thus $\varphi \psi = 0$. \mbox{It follows} that $\mathcal{I} \circ \tau^- \mathcal{I} = 0$ and $\bar{U} \setminus U$ is $\tau^-$-rigid. 

To deduce the desired description of $\mathcal{F}$ we use that $\cogen \bar{U} \setminus U = \cogen \mathcal{I}$, see \Cref{core}, as well as the description of $\mathcal{I}$. For $Z\in \mod A$ the injective hull $Z\to Q$ is contained in $\mathcal{I}$ if and only if $Z\in \cogen U \setminus V$ for all finite subsets $V\subseteq U$ and so $\mathcal{F} = \cogen \bar{U}\setminus U = \cogen \mathcal{I} = \bigcap_V \cogen U \setminus V$.

Next, we show that $\bar{U} \setminus U$ is non-empty. Since ${U}$ is an infinite set and $\Ind A$ is compact \cite{Ziegler} the set $\bar{U}$ contains an accumulation point, which cannot be a finite length module as finite length modules are isolated points, see \cite[Example 3.7]{Krause}. Finally, if the length of $X$ over $\mathrm{End}_A(X)$ is bounded by some $n\in \mathbb{N}$ for $X\in U$ then the length of $Y \in \bar{U}\setminus U$ over $\mathrm{End}_A(Y)$ is bounded by $n$, see \cite[Corollary 9.4]{Herzog}. \linebreak Thus, in this case every $Y \in \bar{U} \setminus U$ is a generic $\tau^-$-rigid module.
\end{proof}

We are now ready to show the equivalence of (2) and (3) in \Cref{strong} over tame algebras, which was first proven in \cite{BPS} using matrix reduction techniques. 

\begin{thm}\label{tame} For a tame finite dimensional algebra $A$ over an algebraically closed field, the following are equivalent. 
\begin{itemize}
    \item[(1)] There are infinitely many non-isomorphic finite dimensional bricks of the same dimension.
    \item[(2)] There exists a generic brick.
\end{itemize}
\end{thm}
\begin{proof} (1)$\implies$(2) By \cite[Proposition 3.4]{MP2} we may assume that there is an infinite hom-orthogonal collection $U$ of non-isomorphic bricks of the same dimension $n$. That is $\Hom_A(X,Y) = 0$ for all $X\neq Y$ in $U$. Over a tame algebra, all but finitely many of these modules are $\tau$-invariant \cite{CB10}. It follows that $U$ fulfills the assumptions in \Cref{generate} and so $\bar{U}\setminus U$ contains a generic $\tau^-$-rigid module, as the length of $X\in U$ over $\mathrm{End}_A(X)$ is bounded by $n$. By \Cref{genericbrick} the existence of a generic $\tau^-$-rigid module implies the existence of a generic brick.

(2)$\implies$(1) Let $k$ be the ground field. By \cite{CB11} every generic $A$-module $X$ is of the form $X\cong {}_AM_{k[t]_f} \otimes_{k[t]_f} k(t)$ such that
\begin{itemize}
    \item[(i)] $M$ is free over $k[t]_f$ of finite rank, 
    \item[(ii)] $\mathrm{End}_A(X)/ \mathrm{rad}\, \mathrm{End}_A(X) \cong k(t)$, and 
    \item[(iii)] $M\otimes_{k[t]_f} k[t]/(t-\lambda)$ for $f(\lambda) \neq 0$ are all non-isomorphic $A$-modules.
\end{itemize}
Here $k[t]_f$ is the ring of polynomials in one variable over $k$ localized at $f\in k[t]$. Assume that $X$ is a brick, so $\mathrm{End}_A(X) \cong k(t)$. Since $M$ is free over $k[t]_f$ of finite rank, it follows that after a finite localization $M_g$ with $g \in k[t]_f$ base change commutes with taking the endomorphism ring. Thus $\mathcal{E} = \mathrm{End}_{A\otimes_k k[t]_{f,g}} (M_{g})$ fulfills $\mathcal{E} \otimes_{k[t]_{f,g}} k(t) \cong \mathrm{End}_{A\otimes_k k(t)} (X) \cong k(t)$ and so 
\begin{align*}
    \mathrm{End}_{A} (M \otimes_{k[t]_{f,g}} k [t]/(t-\lambda)) \cong \mathcal{E} \otimes_{k[t]_{f,g}} k[t]/(t-\lambda) \cong k
\end{align*}
for all but finitely many $\lambda \in k$. This yields the desired family of infinitely many non-isomorphic bricks.
\end{proof}

The Krull-Gabriel dimension, introduced by Geigle \cite{Geigle}, is a measure for the complexity of the module category. We show the full \Cref{strong} under the assumption that the Krull-Gabriel dimension of $A$ is defined. In this case, the Ziegler spectrum has the following nice property.

\begin{lem}\label{kg} If the Krull-Gabriel dimension of $A$ is defined, then every closed subset $U \subseteq \Ind A$ that contains infinitely many finite length modules or an infinite length module must contain a generic module.   
\end{lem}
\begin{proof} See the proof of \cite[Lemma 4.4]{self3}.
\end{proof}

\begin{thm}\label{KG} Let $A$ be a finite dimensional algebra over an algebraically closed field whose Krull-Gabriel dimension is defined. Then $A$ fulfills \Cref{strong}.    
\end{thm}
\begin{proof} Clearly (2) implies (1) in \Cref{strong}. Under the assumptions $A$ is tame, see for example \cite[Proposition 8.15]{Krause}. It follows that (2) and (3) are equivalent by \Cref{tame}. In fact, we only need to use that (3) implies (2). We are left to show that (1) implies (3).

By \cite[Theorem 1.4]{DIJ} the existence of infinitely many non-isomorphic bricks implies that there exists a torsion-free class $\mathcal{F}$ that is not functorially finite. Let $U\subseteq \Ind A$ be the associated support $\tau^-$-tilting subset. Then $U$ contains infinitely many finite length modules or an infinite length module, as otherwise $\mathcal{F} = \cogen U$ would be functorially finite, see \cite[Proposition 4.7]{Auslander0}. By \Cref{kg} the subset $U$ contains a generic module, which is necessarily $\tau^-$-rigid. It follows from \Cref{genericbrick} that $A$ admits a generic brick.
\end{proof}

So far, we have often considered generic bricks and generic $\tau^-$-rigid modules over tame algebras. It is surprising that these classes of modules coincide in this case.

\begin{thm}\label{weird} Let $A $ be a tame finite dimensional algebra over an algebraically closed field. Then, a generic $A$-module $X$ is $\tau^-$-rigid if and only if $X$ is a brick.
\end{thm}
\begin{proof} Over a tame algebra, every generic $A$-module $X$ is the unique accumulation point of infinitely many $\tau$-stable finite dimensional $A$-modules inside $ \Ind A$ and so $X$ is $\tau$-stable by \Cref{homeo}. This was already observed in \cite[Corollary 5.16]{Krause0}. Now, if $X$ is also a brick, then every morphism $X\to X$ with a finite dimensional image must be a non-isomorphism and hence zero. Thus $X$ is $\tau^-$-rigid.

Suppose that $X$ is a generic $\tau^-$-rigid $A$-module and let $S = \mathrm{soc}_{\mathrm{End}_A(X)} X$ be the associated generic brick from the injective assignment in \Cref{inj}. Then $S$ is again $\tau^-$-rigid and the corresponding brick equals  $\mathrm{soc}_{\mathrm{End}_A(S)} S = S$ since $\mathrm{End}_A(S)$ is a division algebra. By injectivity $X\cong S$ and so $X$ is a brick.
\end{proof}

\begin{rem}\rm Let $A$ be a tame finite dimensional algebra over an algebraically closed field $k$. In the proof of (2)$\implies$(1) of \Cref{tame} we have seen that if a generic $A$-module $X$ is a brick, then $M(\lambda)$ is a brick for almost all $\lambda\in k$ where $M$ is the geometric realization of $X$. Slightly adapting the proof of (1)$\implies$(2) shows that if $M(\lambda)$ is a brick for infinitely many $\lambda \in k$ then $X$ is $\tau^-$-rigid and hence a brick by \Cref{weird}. This recovers the main result of \cite{BPS}. One has to use that $X$ is an accumulation point of the $M(\lambda)$ in $\Ind A$, see \cite[Theorem 9.4]{K}.
\end{rem}

\end{document}